\pdfoutput=1
\RequirePackage{ifpdf}
\ifpdf 
\documentclass[pdftex]{sigma}
\else
\documentclass{sigma}
\fi

\numberwithin{equation}{section}

\newtheorem{Theorem}{Theorem}[section]
\newtheorem{Lemma}[Theorem]{Lemma}
\newtheorem{Proposition}[Theorem]{Proposition}
{\theoremstyle{definition}
\newtheorem{Definition}[Theorem]{Definition}
 }

\newcommand{\rk}{\operatorname{r}}

\begin{document}


\newcommand{\arXivNumber}{1504.07165}

\renewcommand{\PaperNumber}{030}

\FirstPageHeading

\ShortArticleName{Polynomial Invariants for Arbitrary Rank~$D$ Weakly-Colored Stranded Graphs}

\ArticleName{Polynomial Invariants for Arbitrary Rank~$\boldsymbol{D}$\\ Weakly-Colored Stranded Graphs}

\Author{Remi Cocou AVOHOU}

\AuthorNameForHeading{R.C.~Avohou}

\Address{International Chair in Mathematical Physics and Applications, ICMPA-UNESCO Chair,\\
072BP50, Cotonou, Republic of Benin}
\Email{\href{avohouremicocou@yahoo.fr}{avohouremicocou@yahoo.fr}}

\ArticleDates{Received June 26, 2015, in f\/inal form March 14, 2016; Published online March 22, 2016}

\Abstract{Polynomials on stranded graphs are higher dimensional generalization of Tutte and
 Bollob\'as--Riordan polynomials [\textit{Math.~Ann.}~\textbf{323} (2002), 81--96]. Here, we deepen the analysis of the polynomial invariant def\/ined on rank~3 weakly-colored stranded graphs introduced in arXiv:1301.1987. We successfully f\/ind in dimension $D\geq3$ a modif\/ied Euler characteristic with $D-2$ parameters. Using this modif\/ied invariant, we extend the rank~3 weakly-colored graph polynomial, and its main properties, on rank~4 and then on arbitrary rank~$D$ weakly-colored stranded graphs.}


\Classification{05C10; 57M15}

\vspace{-2mm}

\section{Introduction}

	The polynomial invariant for $3D$ weakly-colored graphs introduced in~\cite{remia} is a polynomial inva\-riant which extends both the Tutte and Bollob\'as--Riordan (BR) polynomials~\cite{bollo, tutte} from graph and ribbon graphs to a family of combinatorial objects called stranded graphs. This polynomial obeys a particular contraction/cut recurrence relation which replaces the contraction/deletion relation satisf\/ied by Tutte and BR polynomials. Let us review now in greater details the context where these graphs appear and their invariant.

\looseness=-1
The objects that we investigate are called stranded graphs and they still arouse the interest of the mathematicians and physicists \cite{avo,sefu3, Geloun:2009pe,Gurau:2010nd,Gurau:2011xp}. Such stranded graphs are made with stranded vertices which are chord diagrams and stranded edges which are collections of segments. In \cite{Gurau:2010nd}, it has been proved that the colored version of these graphs are dual to simplicial pseudo-manifolds in any dimension $D$. Restricted to $2D$, the same class of graphs corresponds to ribbon graphs or rank~2 stranded graphs with vertices of coordination~3. Considering the ribbon vertices as $0$-cells, lines or edges as $1$-cells and faces as $2$-cells, ribbon graphs become two dimensional cellular complexes. However, the cellular complex structure has been not yet generalized for all stranded graphs in higher dimensions but only on those which are colored~\cite{Gurau:2009tw}. In a colored graph, a $p$-cell or $p$-bubble is def\/ined as a connected subgraph made only of lines of~$p$ chosen colors.

Colored tensor graphs are specif\/ic stranded graphs \cite{Gurau:2010nd} on which we impose a f\/ixed coordination of vertices and an edge coloring. The contraction of an edge modif\/ies this coordination and then destroys the color structure. This has led to the enlargement of the family of graphs from the colored tensor graphs to what will be called weakly-colored (w-colored) stranded graphs for which contraction and a notion similar to the deletion make sense.

Stranded vertices are furthermore decorated with half-edges called ``half lines'' (or external legs in quantum f\/ield theory framework \cite{Gurau:2009tz, Gurau:2009tw}). The graphs with decorated vertices are very useful in physics. For instance, it was def\/ined on these graphs a~generalized Symanzik polynomial with very nice properties in quantum f\/ield theory due to the importance of the external legs~\cite{riv}. The notion of half-edge~\cite{tutte} is consistent with the operation that involves cutting an edge of a~given graph. The cut of an edge of a~graph means removing this edge and let two half-edges attached to its incidence vertices (or vertex in the loop situation). Moreover, this operation brings some modif\/ications on the ``boundary of the graph''. The boundary graph encodes the boundary of the simplicial complex dual to the stranded graph. We now distinguish two kinds of faces: closed strands, homeomorphic to~$S_1$ which def\/ine the internal faces and the remaining which are homeomorphic to $(0,1)$ def\/ine the external faces.

The generalization of the universal invariant, namely Tutte polynomial, from graphs to ribbon graphs \cite{bollo3,bollo2,bollo,ellis0} was performed by Bollob\'as and Riordan by adding two supplementary variables to the Tutte polynomial. One variable stands for the orientability of the ribbon seen as surface and the other one for an invariant related to the genus of the corresponding ribbon graph.

The Bollob\'as--Riordan polynomial invariant was generalized by Tanasa in \cite{Tanasa:2010me} where one supplementary variable keeps track of the sum of genera \cite{Gurau:2012ix, Ryan} of the $3$-bubbles (surfaces) of a rank~3 tensor graph. A similar procedure proves to be applicable for f\/inding an invariant on rank 3 w-colored graphs. We emphasize that, several other potentially interesting invariant candidates under contraction and cut can be identif\/ied. In $D\geq 4$, we must expect a much richer structure of the possible invariants.

In this paper, we discover several possibilities when we deal with $D\geq 4$. We provide two combinatorial methods to generate a generalized Euler characteristic which can be used to def\/ine an extension of the invariant in rank~3 determined in~\cite{remia}. Furthermore, for each choice of invariant, it appears possible to parametrize the invariant by a f\/inite family of positive rational numbers. As a consequence, the polynomial invariant built on these invariants will acquire new parameters $\alpha_k\in \mathbb{Q}^{+}$, $k=3, \dots, n$. The modif\/ied Euler characteristic $\gamma_{n;\boldsymbol{\alpha}}$ of a rank $n\geq 3$ weakly-colored graph $\mathcal{G}$ is given by
\begin{gather*}
\gamma_{n;\boldsymbol{\alpha}}=\frac{n(n-1)}{2}(V-E) + (n-1)F_{{\rm int}}- (2+(n-2)\alpha_3)B^3
\\
\hphantom{\gamma_{n;\boldsymbol{\alpha}}=}{} + \sum_{k=4}^n \big[(k-1)\alpha_{k-1}-(n-k+1)\alpha_k\big]B^k,
\end{gather*}
where $V$, $E$, $F_{{\rm int}}$ and $B^k$ are respectively the number of vertices, edges, internal faces and $k$-bubbles of $\mathcal{G}$. Our f\/irst main result stated in Theorem~\ref{theo:contens4} describes the relations satisf\/ied by the polynomial invariant in rank~4. By induction, we determine a polynomial invariant on arbitrary rank $D$ w-colored graphs. Our second main result is the contraction/cut recurrence relation (Theorem~\ref{theo:contensn}) satisf\/ied by this polynomial. The usual Tutte and BR polynomials can be recovered after restricting to the proper graph theory rank and after an appropriate change of variables.

The organization of the rest of the paper is the following. In the next section, we give a brief review of the polynomial invariant on rank 3 w-colored stranded graphs. Section \ref{subset:polymt} is dedicated to the extension of this polynomial on rank 4 w-colored stranded graphs. By an induction procedure we generalize in Section~\ref{subset:polymtn}, the rank 4 polynomial to arbitrary rank $D$ w-colored stranded graphs. Finally, an appendix provides a detailed analysis of various other interesting invariants which could be alternatively used for def\/ining new polynomials extending Tutte and BR in any dimension $D$.

\section{Weakly-colored graphs and the rank 3 polynomial invariant}
\label{sect:backg}
We brief\/ly recall here the def\/inition of the polynomial invariant def\/ined for rank~3 w-colored graphs introduced in \cite{remia}. We also address its main properties which will f\/ind extension on higher rank graphs.

A graph $\mathcal{G}(\mathcal{V},\mathcal{E})$ is def\/ined as a set of vertices $\mathcal{V}$ and of edges $\mathcal{E}$ together with an incidence
relation between them. $\mathcal{G}(\mathcal{V},\mathcal{E})$ is stranded when its vertices and edges are stranded.

\begin{Definition}[stranded vertex and edge]\label{def:svert}
A rank $D$ stranded vertex is a chord diagram
that is a collection of $2n$ points on the unit circle (called the vertex
frontier) paired by $n$ chords, satisfying:
\begin{enumerate}\itemsep=0pt
\item[(a)] the chords may cross, but do not intersect;
\item[(b)] the chord end points are partitioned
 into sets called pre-edges with $0,1,2,\dots$, or~$D$ elements;
these points should lie on a single arc on the frontier with
no other end points on this arc;
\item[(c)] the pre-edges should form a connected collection
that is, by merging all points in each pre-edge and
by removing the vertex frontier, the resulting
graph is connected.
\end{enumerate}

The coordination (also called valence or degree) of a rank $D >0$ stranded vertex
is the number of its non-empty pre-edges.
By convention: (C1)~we include a particular vertex made with one
 disc and assume that it is a stranded vertex of any rank made
 with a unique closed chord and~(C2) a point is a rank 0 stranded vertex.

A rank $D$ stranded edge is a collection of segments called strands such that:
\begin{enumerate}\itemsep=0pt
\item[(a$'$)] the strands are not intersecting (but
can cross without intersecting);
\item[(b$'$)] the end points of the strands can be
partitioned in two disjoint parts called sets of end segments of the
edge such that a strand cannot have its end points
in the same set of end segments;
\item[(c$'$)] the number of strands is~$D$.
\end{enumerate}
\end{Definition}

Some illustrations of stranded vertices and edges with rank $D=4,5$, respectively, are provided in Fig.~\ref{fig:vxedge}.
\begin{figure}[h]
 \centering
\includegraphics[angle=0, width=10cm, height=3cm]{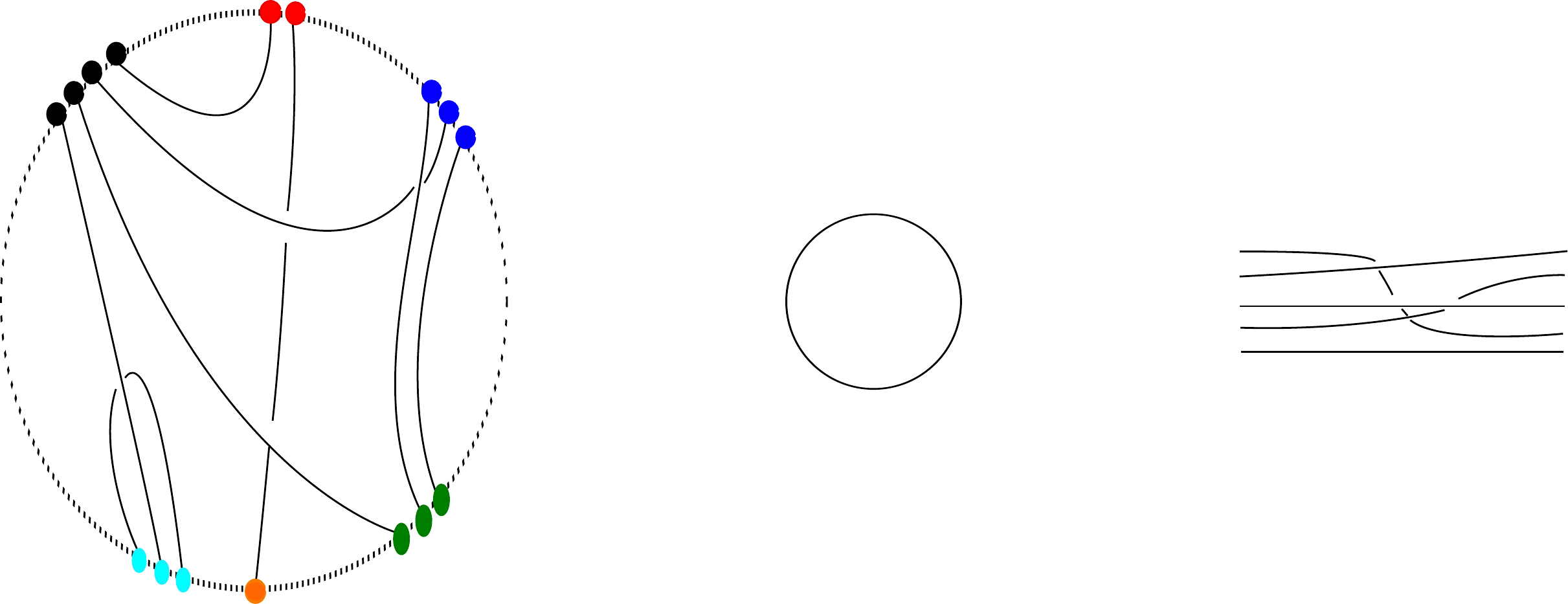}
\put(-240,-10){$v$}
\put(-130,15){$d$}
\put(-35,15){$e$}

\caption{A rank 4 stranded vertex $v$ of coordination~6, connected pre-edges highlighted with dif\/ferent colors; a trivial disc vertex~$d$; a rank~5 edge~$e$ with non parallel strands.}\label{fig:vxedge}
\end{figure}

\begin{Definition}[stranded and tensor graphs]\label{def:tens}
A rank $D$ stranded graph $\mathcal{G}$ is a graph $\mathcal{G}(\mathcal{V},\mathcal{E})$
which admits:
\begin{enumerate}\itemsep=0pt
\item[(i)] rank $D$ stranded vertices;
\item[(ii)] rank at most $D$ stranded edges;
\item[(iii)] one vertex and one edge intersect
by one set of end segments of the edge which should coincide with a~pre-edge at the vertex frontier;
All intersections of vertices and edges are pairwise distinct.
\end{enumerate}

A rank $D$ tensor graph $\mathcal{G}$ is a rank $D$ stranded graph such that:
\begin{enumerate}\itemsep=0pt
\item[(i$'$)] the vertices of $\mathcal{G}$ have a f\/ixed coordination $D+1$
and their pre-edges have a f\/ixed cardinal~$D$. The merged point graph is $K_{D+1}$ (or vertex graph);
\item[(ii$'$)] the edges of~$\mathcal{G}$ are of rank~$D$.
\end{enumerate}
\end{Definition}

Consider a stranded graph. Collapsing its stranded vertices to points and edges to simple lines, the resulting object is a graph. A stranded graph is said to be connected if its correspon\-ding collapsed graph is connected. From this point, stranded vertices and edges are claimed connected. 		

Before proceeding further, let us present the notion of colorable (tensor) graphs~\cite{Gurau:2010nd}.

\begin{Definition}[colored and bipartite graphs]\label{def:colbi}
 A $(D+1)$ colored graph is a graph
together with an assignment of a color belonging to the set $\{0,\dots,D\}$ to each of its edges
such that no two adjacent edges share the same color.

A bipartite graph is a graph whose set $\mathcal{V}$ of vertices
is split into two disjoint sets, i.e., $\mathcal{V} = \mathcal{V}^+ \cup \mathcal{V}^-$
with $\mathcal{V}^+\cap\mathcal{V}^-= \varnothing$, such that each edge
connects a vertex $v^+ \in \mathcal{V}^+$ and a vertex
$v^- \in \mathcal{V}^-$.
\end{Definition}

\begin{Definition}[colored tensor graph~\cite{Gurau:2009tw,Gurau:2011xp}]\label{def:coltens}
A rank $D\geq 1$ colored tensor graph $\mathcal{G}$ is a graph such that:
\begin{itemize}\itemsep=0pt
\item $\mathcal{G}$ is $(D+1)$ colored and bipartite;
\item $\mathcal{G}$ is a rank $D$ tensor graph.
\end{itemize}
\end{Definition}

The collapsed graph coming from a colored tensor graph is merely obtained by regarding the tensor graph as a simple bipartite colored graph and is called \emph{compact} in the following.

Certainly, in such a colored graph, there is a lot more information that we address now.

\begin{Definition}[$p$-bubbles \cite{Gurau:2009tw}]
\label{def:pbub}
Let $\mathcal{G}$ be a rank $D$ colored tensor graph.
\begin{itemize}\itemsep=0pt
\item[--] A 0-bubble is a vertex of $\mathcal{G}$.

\item[--] A 1-bubble is an edge of $\mathcal{G}$.

\item[--] For all $p \geq 2$, a $p$-bubble of $\mathcal{G}$ with colors $i_1 < i_2 < \cdots<i_p $, $p\leq D$,
and $i_k\in \{0, \dots, D\}$ is a connected rank $p-1$ colored tensor graph the compact form of which is a connected subgraph
of the compact form of $\mathcal{G}$ made of edges of colors $\{i_1, \dots, i_p\}$.
\end{itemize}
 \end{Definition}

	The $p$-bubbles must not be confused with the notion of subgraphs used in the def\/inition of the polynomial invariants. They are basic components of combinatorial graphs and generate an associated homology~\cite{Gurau:2009tw}. A 2-bubble is also called a face.

We now introduce the notion of half-edge which allows to address the operation called ``cut'' of an edge. This operation is dif\/ferent from the usual edge deletion and it is used to def\/ine another category of subgraphs called ``c-subgraphs''.

\begin{Definition}[rank~$D$ stranded half-edge]\label{def:tensflag}
A rank $D$ stranded half-edge is a collection of
$D$ parallel segments called strands satisfying the same properties of strands of rank~$D$ edges but the stranded half-edge is incident to a unique rank~$D' $ stranded vertex, with $D' \geq D$, by one of its set of end segments without forming a loop.

A rank~$D$ stranded half-edge has two sets of end segments: one touching a vertex
and another called free or external set of end segments,
the elements of which are called themselves free or external segments.
The~$D$ end-points of all free segments are called external points
of the rank~$D$ stranded half-edge.
\end{Definition}

A stranded graph having stranded half-edges is called half-edged stranded graph and denoted $\mathcal{G}(\mathcal{V},\mathcal{E},\mathfrak{f}^0)$ or simply $\mathcal{G}_{\mathfrak{f}^0}$, with $\mathfrak{f}^0$ the set of the half-edges.

\begin{Definition}[cut of an edge \cite{riv}]\label{def:cutedtens}
Let $\mathcal{G}(\mathcal{V},\mathcal{E},\mathfrak{f}^0)$ be a rank D stranded graph
and~$e$ a rank~$d$ edge of $\mathcal{G}_{\mathfrak{f}^0}$, $1\leq d \leq D$.
The cut graph $\mathcal{G}_{\mathfrak{f}^0} \vee e$ or the graph obtained from $\mathcal{G}_{\mathfrak{f}^0}$ by cutting $e$ is obtained by replacing the edge $e$ by two rank $d$ stranded half-edges at the end vertices of~$e$ and respecting the strand structure
of~$e$, see Fig.~\ref{fig:cuttens}. If~$e$ is a loop, the two stranded half-edges are on the same vertex.
\end{Definition}

\begin{figure}[h]
 \centering
\includegraphics[angle=0, width=5cm, height=1cm]{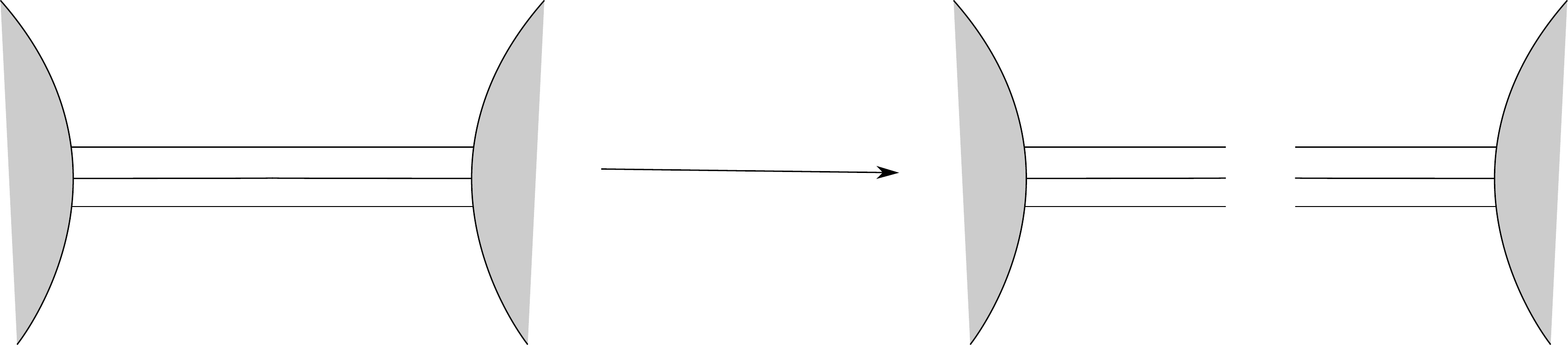}
\caption{Cutting a rank 3 stranded edge.}
\label{fig:cuttens}
\end{figure}

Rank D stranded half-edges can be considered on colored tensor graphs by respecting the following condition in addition to Def\/inition~\ref{def:tens}: to each edge and stranded half-edges, one assigns a color $i \in \{0,1,\dots,D\}$ such that no two adjacent stranded half-edges share the same color. As well, the colored edge can be cut in the same sense of Def\/inition~\ref{def:cutedtens}.
The crucial issue is to respect the color structure of the graph after the cut such that each of the resulting stranded half-edges possesses the same color structure of the former edge. Rank~$D$ half-edged colored tensor graph
are rank $D$ half-edged tensor graph equipped with edge and stranded half-edge coloring.

Let us come back on the notion of c-subgraph using the operation of cutting of an edge. We obtain a c-subgraph $A(\mathcal{V}_A,\mathcal{E}_A,\mathfrak{f}^0_A)$ of a rank $D$ stranded graph $\mathcal{G}(\mathcal{V},\mathcal{E},\mathfrak{f}^0)$ by cutting a subset of edges of $\mathcal{G}_{\mathfrak{f}^0}$. A spanning c-subgraph $A$ of $\mathcal{G}_{\mathfrak{f}^0}$ is def\/ined as a c-subgraph $A(\mathcal{V}_A,\mathcal{E}_A,\mathfrak{f}^0_A)$ of $\mathcal{G}_{\mathfrak{f}^0}$ with all vertices and all additional half-edges of $\mathcal{G}_{\mathfrak{f}^0}$. Hence $\mathcal{E}_A\subseteq \mathcal{E}$ and $\mathcal{V}_A = \mathcal{V}$, $\mathfrak{f}^0_A = \mathfrak{f}^{0} \cup \mathfrak{f}^{0;1}_A(\mathcal{E}_A)$, where $\mathfrak{f}^{0;1}_A (\mathcal{E}_A)$ is the set of half-edges obtained by cutting all edges in~$\mathcal{E}_A'$ (the set of edges incident to the vertices of~$A$
and not contained in $\mathcal{E}_A$) and incident to vertices of~$A$. We denote it $A \Subset \mathcal{G}_{\mathfrak{f}^0}$.

This operation, which involves cutting an edge, has many implications on the structure of the graph. For example it modif\/ies the strand structure of this graph. In particular, the cut operation or the presence of half-edges immediately introduce another type of faces which pass through the external points of the half-edges. Combinatorially, a discrepancy is introduced between this type of faces called open faces and those
which do not pass through the external points of the half-edges are closed. The sets of closed and open faces is denoted by~$\mathcal{F}_{{\rm int}}$
and~$\mathcal{F}_{{\rm ext}}$, respectively. Hence, for a rank $D$ half-edged colored tensor graph, the set $\mathcal{F}$ of faces is the disjoint union~$\mathcal{F}_{{\rm int}} \cup \mathcal{F}_{{\rm ext}}$. The notion of closed or open rank~$D$ half-edged colored tensor graph can be reported accordingly if $\mathcal{F}_{{\rm ext}}=\varnothing$ or not, respectively. A bubble is open or external if it contains open faces otherwise it is closed or internal. The sets of closed and open bubbles for rank 3 tensor graph is denoted by $\mathcal{B}_{{\rm int}}$ and $\mathcal{B}_{{\rm ext}}$, respectively (see an illustration in Fig.~\ref{fig:tenfla}).

\begin{figure}[h]
 \centering

\includegraphics[angle=0, width=12cm, height=2.5cm]{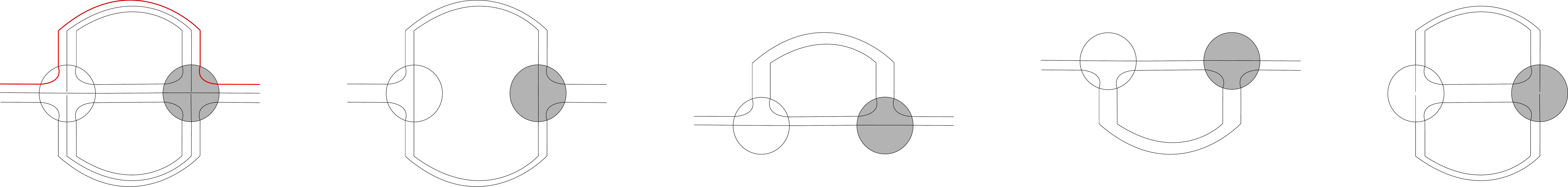}
\put(-350,33){1}
\put(-316,-10){2}
\put(-316,45){3}
\put(-316,75){0}

\caption{An open rank 3 half-edged colored tensor graph; an open face highlighted in red; the bubbles
${\mathbf{b}}_{012}$, ${\mathbf{b}}_{013}$ and ${\mathbf{b}}_{123}$ are open bubbles and ${\mathbf{b}}_{023}$ is closed.}
\label{fig:tenfla}
\end{figure}
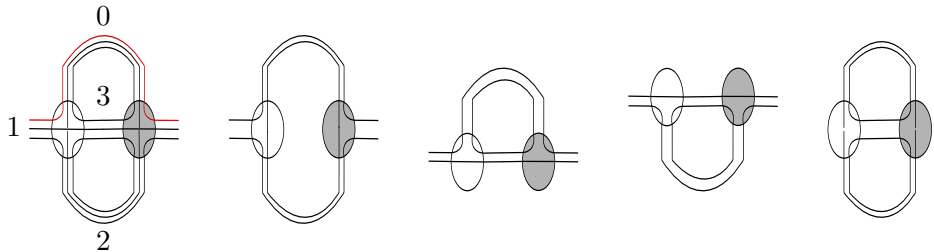

The presence of half-edges produces a new graph called ``boundary graph'' which is obtained by setting a vertex to each half-edge \cite{Gurau:2009tz}.

\begin{Definition}[boundary tensor graph \cite{Gurau:2009tz}]
\label{def:bgph}
The boundary graph $\partial{\mathcal{G}}({\mathcal{V}}_{\partial},{\mathcal{E}}_{\partial})$ of a rank~$D$
half-edged colored tensor graph $\mathcal{G}(\mathcal{V},\mathcal{E},\mathfrak{f}^0)$ is a graph obtained by inserting a vertex
with degree~$D$ at each additional stranded half-edge of $\mathcal{G}_{\mathfrak{f}^0}$ and taking the external
faces of $\mathcal{G}_{\mathfrak{f}^0}$ as its edges. Thus, $|{\mathcal{V}}_{\partial}| = |\mathfrak{f}^0|$ and
${\mathcal{E}}_{\partial} = \mathcal{F}_{{\rm ext}}$.

The boundary of a closed rank $D$ half-edged colored tensor graph is empty.
\end{Definition}

\begin{figure}[h] \centering

\includegraphics[angle=0, width=3cm, height=2cm]{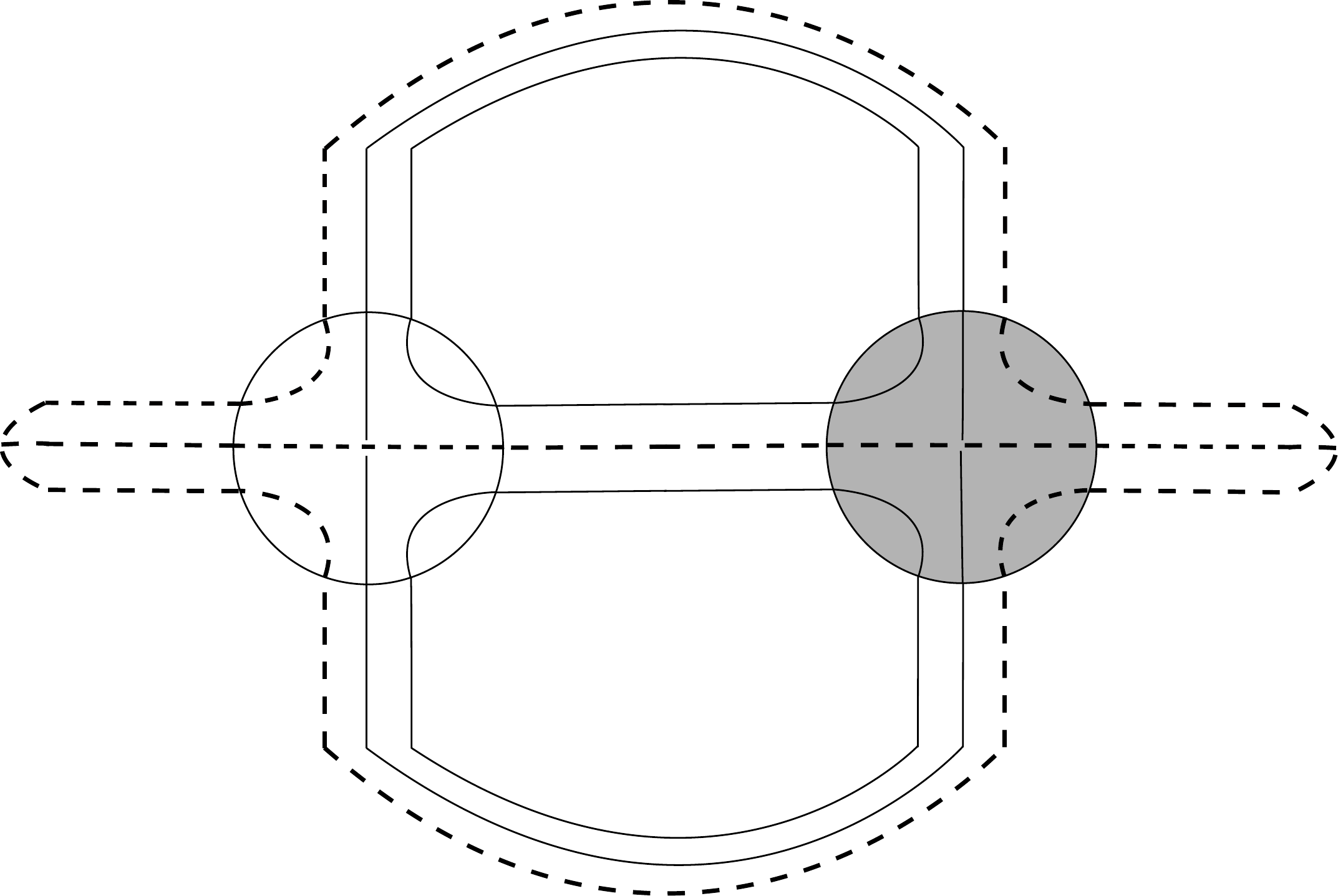}
\put(-80,33){1}
\put(-10,33){1}
\put(-45,-10){2}
\put(-45,35){3}
\put(-45,60){0}

\caption{Boundary graph (in dashed lines) of the graph of Fig.~\ref{fig:tenfla}.}
\label{fig:bound3}
\end{figure}
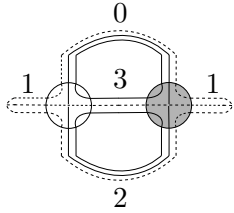

In order to introduce the notion of contraction of a stranded edge, some particular edges called $p$-inner edge (an illustration is given in Fig.~\ref{fig:mpinner4}) must be discussed. Assuming that $e$ is of rank~$D$, consider the two pre-edges $f_1$ and $f_2$ where $e$ is branched to its end vertex~$v$ (a~loop situation) or vertices~$v_{1,2}$ (a~non loop case). It may happen that, after branching $e$, $p$ closed faces are formed such that these closed faces are completely contained in~$e$ and~$v$ or~$v_{1,2}$. These closed faces are called inner faces of the edge. An edge with~$p$ inner faces is called $p$-inner edge.

The notion of contraction introduced in this work is the notion of ``soft'' contraction in~\cite{remia}.

\begin{figure}[h]
 \centering

\includegraphics[angle=0, width=10cm, height=2.2cm]{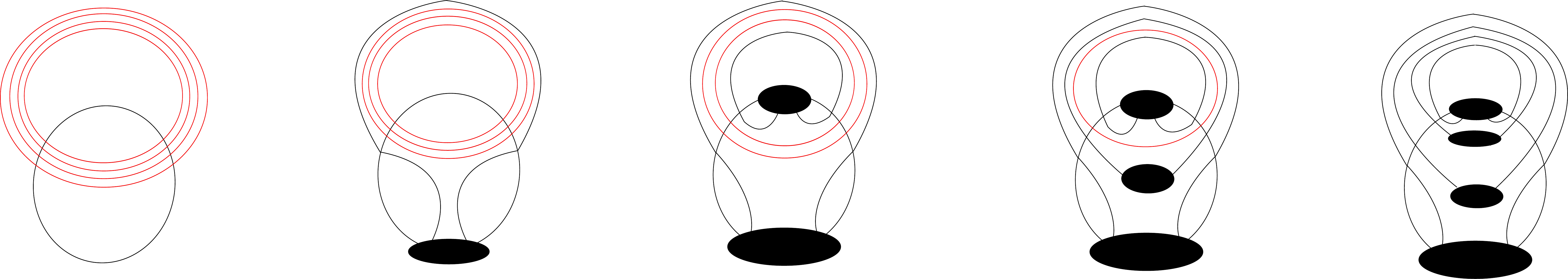}
\put(-271,-19){O}
\put(-207,-19){A}
\put(-146,-19){B}
\put(-80,-19){C}
\put(-20,-19){D}
\put(-146,45){$v_1$}
\put(-146,14){$v_2$}
\put(-80,45){$v_1$}
\put(-80,14){$v_2$}
\put(-80,-7){$v_3$}
\put(-20,43){$v_1$}
\put(-20,23){$v_2$}
\put(-20,11){$v_3$}
\put(-20,-7){$v_4$}

\caption{$p$-inner loops: a 4-inner~(O)
and a 3-inner~(A) loop with their unique possible conf\/iguration; a~2-inner loop with possible two sectors~$v_1$ and~$v_2$~(B) and a~1-inner loop with three possible sectors~$v_1$,~$v_2$ and~$v_3$~(C) and a 0-inner loop with four possible sectors $v_1,v_2,v_3$ and $v_4$~(D).}
\label{fig:mpinner4}

\end{figure}
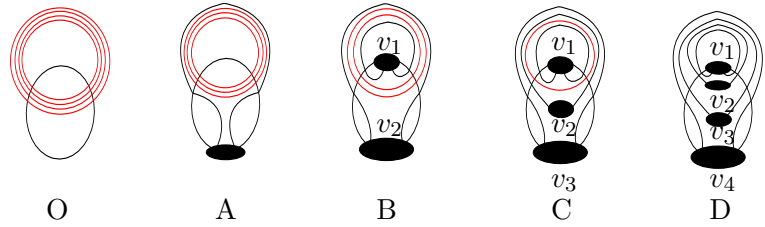

\begin{Definition}[stranded edge contractions]\label{def:constran}
Let $\mathcal{G}(\mathcal{V},\mathcal{E},\mathfrak{f}^0)$ be a rank $D$ half-edges stranded graph. Let $e$ be a rank~$D$ edge with
pre-edges $f_1$ and $f_2$ and consider their neighbor families~$\{f_{1;i}\}$ and~$\{f_{2;j}\}$.

If $e$ is a rank $D$ $p$-inner edge but not a loop, the graph $\mathcal{G}_{\mathfrak{f}^0}/e$ obtained by contracting $e$ is def\/ined from $\mathcal{G}_{\mathfrak{f}^0}$ by removing the $p$ inner faces generated by $e$,
 replacing $e$ and its end vertices $v_{1,2}$ by $p$ disjoint disc vertices and a new vertex $v'$.
The new vertex $v'$ possesses all pre-edges except for those of $e$ and all stranded half-edges as they appear on $v_{1}$ and $v_{2}$, and chords obtained by connecting directly $\{f_{1;i}\}$ and/or $\{f_{2;j}\}$ via the outer strands of $f_1$ and $f_2$.

If $e$ is a rank $D$ $p$-inner loop, the graph $\mathcal{G}_{\mathfrak{f}^0}/e$ obtained by contracting $e$ is def\/ined from~$\mathcal{G}_{\mathfrak{f}^0}$ by removing the $p$ inner faces of $e$, by replacing $e$ and its end vertex $v$ by $p$ disjoint disc vertices and one vertex $v'$ having all pre-edges of $v$ except for those of $e$, all stranded half-edges and chords built in the similar way as previously done by connecting the neighbor families of $f_1$
and/or $f_2$. If there is no outer strand left after removing the $p$ inner faces of $e$ then the vertex~$v'$ is empty. (Examples of rank 4 loop contraction is given in Fig.~\ref{fig:wpcont4}.)
\end{Definition}

\begin{figure}[h]
 \centering

\includegraphics[angle=0, width=9cm, height=1.5cm]{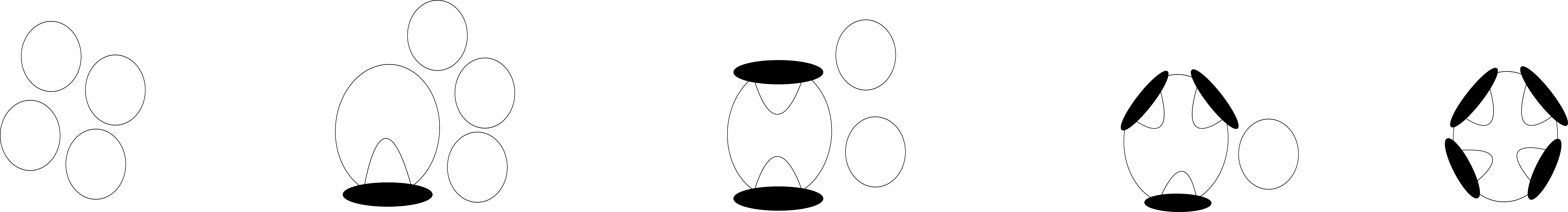}
\put(-247,-19){O}
\put(-189,-19){A}
\put(-132,-19){B}
\put(-67,-19){C}
\put(-12,-19){D}
\put(-132,33){$v_1$}
\put(-132,-7){$v_2$}
\put(-80,25){$v_1$}
\put(-56,25){$v_2$}
\put(-67,-7){$v_3$}
\put(-27,25){$v_1$}
\put(-2,25){$v_2$}
\put(-27,1){$v_3$}
\put(-2,1){$v_4$}

\caption{$p$-inner loop contractions corresponding to O, A, B, C and D of Fig.~\ref{fig:mpinner4}, respectively.}
\label{fig:wpcont4}
\end{figure}
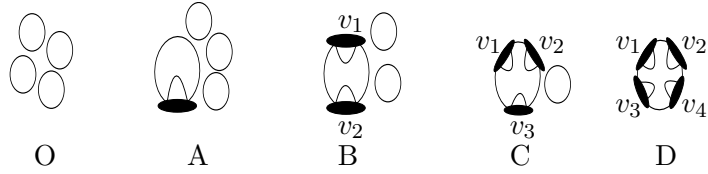

\begin{Proposition}
Let $\mathcal{G}(\mathcal{V},\mathcal{E},\mathfrak{f}^0)$ be a rank $D$ half-edged stranded graph and $e$ be one of its edges. The graph $\mathcal{G}_{\mathfrak{f}^0}/e$ obtained by contraction of~$e$ admits a rank~$D$ half-edged stranded graph structure.
\end{Proposition}

\begin{Definition}[equivalence class of half-edged stranded graph]
Let $D_{\mathcal{G}_{\mathfrak{f}^0}}$ be the subgraph in a rank $D$ half-edges stranded graph $\mathcal{G}_{\mathfrak{f}^0}$ def\/ined by all of its trivial disc vertices and $\mathcal{G}_{\mathfrak{f}^0} {\setminus} D_{\mathcal{G}_{\mathfrak{f}^0}}$ the rank $D$ half-edges stranded graph obtained after removing $D_{\mathcal{G}_{\mathfrak{f}^0}}$ from $\mathcal{G}_{\mathfrak{f}^0}$.

Two rank $D$ half-edged stranded graphs $\mathcal{G}_{1,\mathfrak{f}^0(\mathcal{G}_1)}$ and $\mathcal{G}_{2,\mathfrak{f}^0(\mathcal{G}_2)}$ are ``equivalent up to trivial discs'' if and only if
$ \mathcal{G}_{1,\mathfrak{f}^0(\mathcal{G}_1)}{\setminus} D_{ \mathcal{G}_{1,\mathfrak{f}^0(\mathcal{G}_1)}} = { \mathcal{G}_{2,\mathfrak{f}^0(\mathcal{G}_2)}}{\setminus} D_{ \mathcal{G}_{2,\mathfrak{f}^0(\mathcal{G}_2)}}$.
We note $ \mathcal{G}_{1,\mathfrak{f}^0(\mathcal{G}_1)} \sim { \mathcal{G}_{2,\mathfrak{f}^0(\mathcal{G}_2)}}$.
\end{Definition}

\begin{Lemma}[full contraction of a tensor graph]
\label{lem:fullcont}
Contracting an edge in a rank $D$ $($colored$)$ half-edged tensor graph does not change its boundary.
The contraction of all edges in arbitrary order of a half-edged tensor graph~$\mathcal{G}_{\mathfrak{f}^0}$ $($possibly with colors$)$ yields a half-edged stranded graph~$\mathcal{G}^0_{\mathfrak{f}^0}$ determined by the boundary $\partial( \mathcal{G}_{\mathfrak{f}^0})$ up to additional discs.
\end{Lemma}

Based on this lemma we can now introduce the def\/inition of a rank $D$ w-colored graph.
An example is given in Fig.~\ref{fig:contens}.

\begin{Definition}[rank $D$ w-colored graph]\label{wdcolo}
A rank $D$ weakly-colored or w-colored graph is the equivalence class (up to trivial discs) of a rank $D$ half-edged stranded graph obtained by successive edge contractions of some rank $D$ half-edged colored tensor graph.
\end{Definition}

\begin{figure}[h]
 \centering

\includegraphics[angle=0, width=7cm, height=2cm]{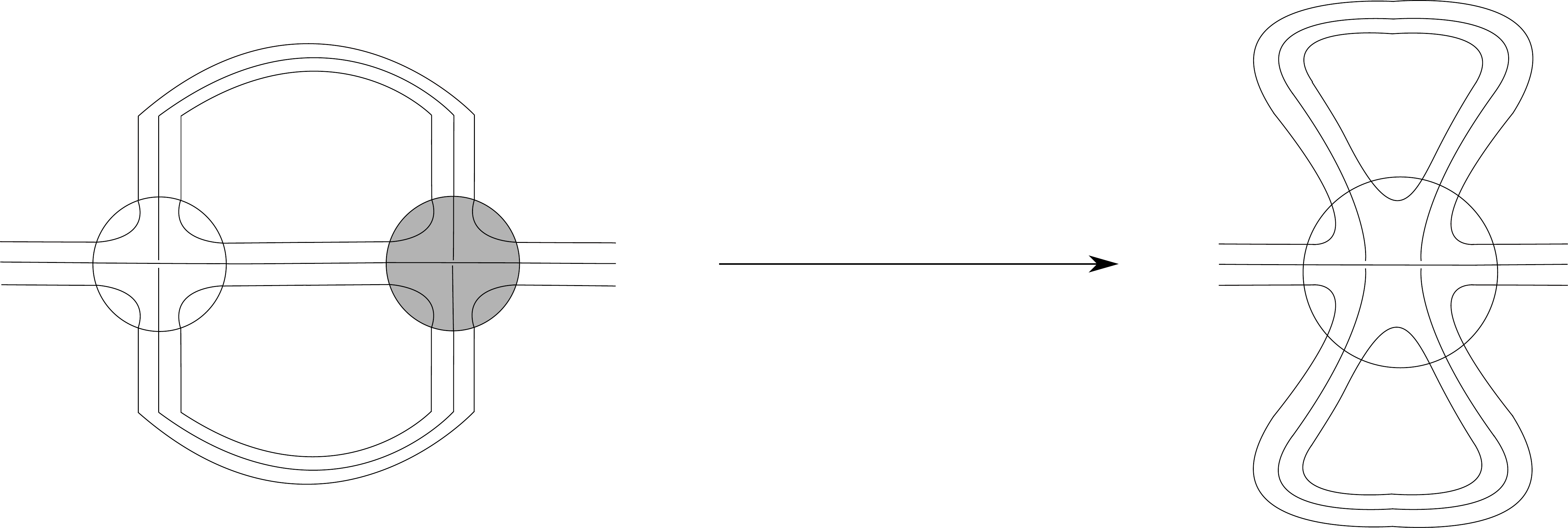}
\put(-118,25){1}
\put(-208,25){1}
\put(-163,-5){2}
\put(-163,33){3}
\put(-163,55){0}
\put(-53,25){1}
\put(2,25){1}
\put(-25,-10){2}
\put(-25,60){0}

\caption{Contraction of an edge in a rank 3 half-edged colored tensor graph.}\label{fig:contens}
\end{figure}
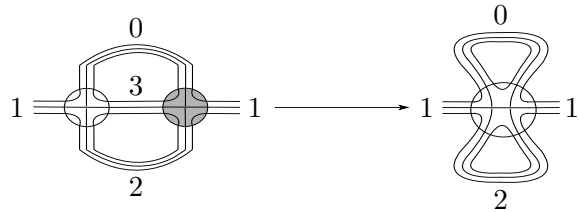

\begin{Proposition} \label{lem:cutmg}
Let $\mathfrak{G}$ be a rank $D$ w-colored graph,
$\mathcal{G}_{\mathfrak{f}^0}$ be a representative of~$\mathfrak{G}$ and
$e$ one of its edges.

$\mathfrak{G} \vee e$ denotes the equivalence class
of $\mathcal{G}_{\mathfrak{f}^0}\vee e$ called the cut graph~$\mathfrak{G}$ along~$e$
and is a~rank~$D$ w-colored graph.

$\mathfrak{G}/e$ denotes the equivalence class of $\mathcal{G}_{\mathfrak{f}^0}/e$,
called the contraction of $\mathfrak{G}$ along~$e$
and is a~rank~$D$ w-colored graph.
 \end{Proposition}

Let us discuss the notion of trivial loop for $4D$ w-colored stranded graphs.
A loop $e$ is called trivial if it is a 4-inner or 3-inner or if it is a 2-inner, 1-inner or 0-inner loop with all separate sectors such that there is no edge between sectors $v_i$ as illustrated in Fig.~\ref{fig:mpinner4}.

For a 4-inner loop, the contraction gives four trivial
discs, see Fig.~\ref{fig:wpcont4}O. For a 3-inner loop, the contraction yields to
Fig.~\ref{fig:wpcont4}A. For a trivial 2-inner loop the contraction is still straightforward and yields Fig.~\ref{fig:wpcont4}B. Contracting a trivial 2-inner loop, the vertex gets disconnected in two non trivial vertices. For a 1-inner loop contraction (see Fig.~\ref{fig:wpcont4}C), one gets one extra disc and has two possible conf\/igurations: either the vertex remains connected or it gets disconnected with three (possibly non trivial) vertices in both situations. If the 1-inner loop is trivial, it is immediate that the vertex gets disconnected in three non trivial vertices. For a 0-inner loop contraction (see Fig.~\ref{fig:wpcont4}D), we have no additional disc but four types of conf\/igurations with up to four disconnected (and possibly non trivial) vertices. The contraction of a trivial 0-inner loop yields directly four disconnected and non trivial vertices.

The notion of trivial loop for $3D$ w-colored is direct from the above discussion: a rank 3 loop $e$ is called trivial if it is a 3-inner or 2-inner or if it is a 1-inner or 0-inner with all separate sectors such that there is no edge between sectors $v_i$.

{\bf Equivalence class of stranded graphs.}
We now def\/ine the equivalence class of stranded graphs by performing a sequence of operations corresponding to the following:
\begin{enumerate}\itemsep=0pt
\item[--] any homeomorphism of chords and strands keeping f\/ixed their end points
(we therefore use a ``minimal'' graphical representation for stranded vertices and edges which is the one def\/ined by chords and strands using simple arcs between the pre-edge points);

\item[--] any change of the crossing states between chords and
also between strands (see Fig.~\ref{fig:cross});

\begin{figure}[h]
 \centering
\includegraphics[angle=0, width=4cm, height=2.5cm]{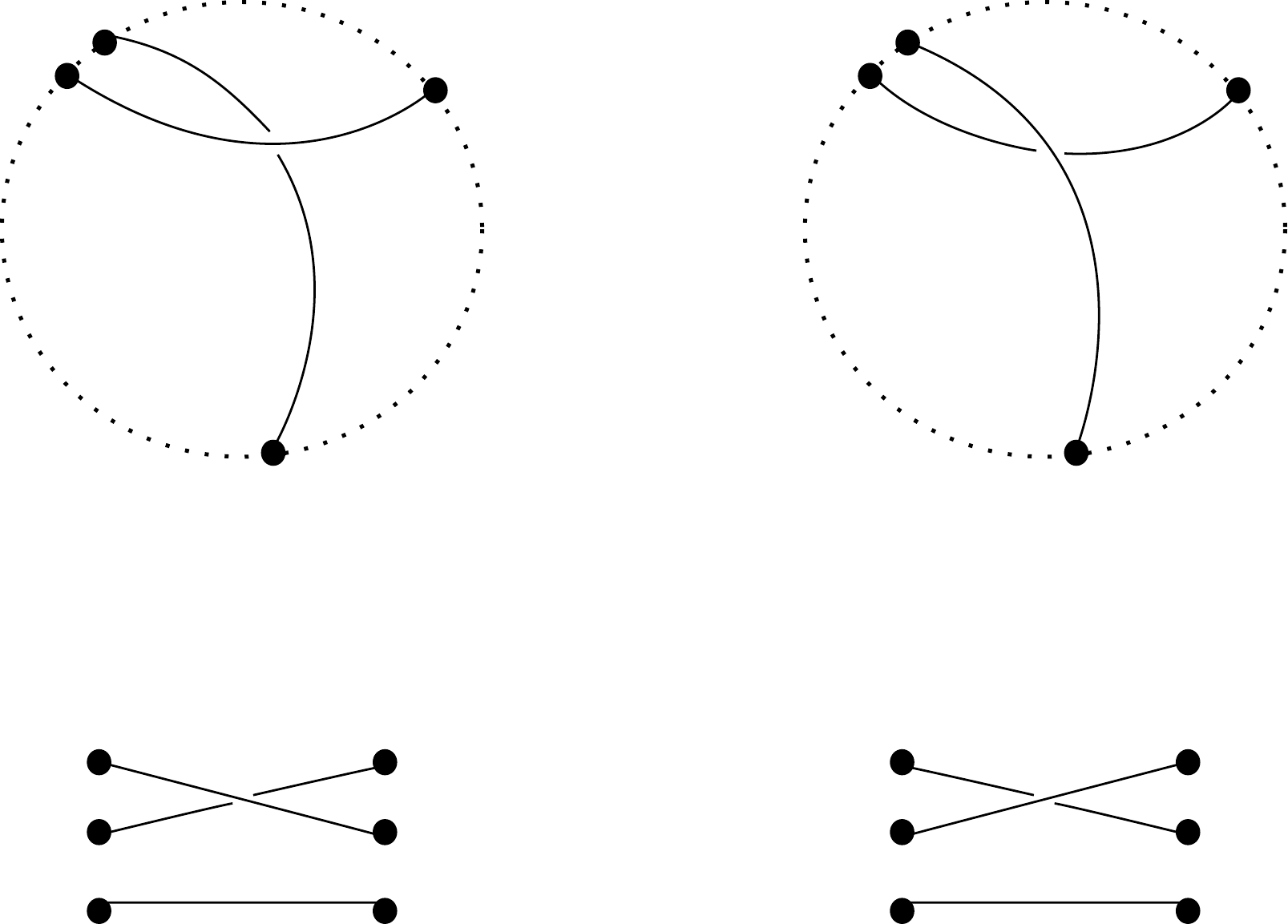}
\caption{Examples of equivalent crossing states of chords (in stranded
vertices) and strands (in stranded edge).}\label{fig:cross}
\end{figure}

\item[--] any permutation of the pre-edge points within a pre-edge
(see Fig.~\ref{fig:samevert}: $v_1$ is obtained from $v_0$ after permuting the points
1 and 2 in the pre-edge $\{1,2,3\}$ and renaming $e_1 \to \tilde{e}_1$).

\item[--] any move of the pre-edges on the frontier vertex
(see Fig.~\ref{fig:samevert}: $v_2$ is obtained from $v_0$ after moving a pre-edge $\{7,8\}$ on the vertex frontier
and also its edge $e_3$ remains incident
to that pre-edge at the same pre-edge points).
\end{enumerate}

\begin{figure}[h]
 \centering

\includegraphics[angle=0, width=8cm, height=7.5cm]{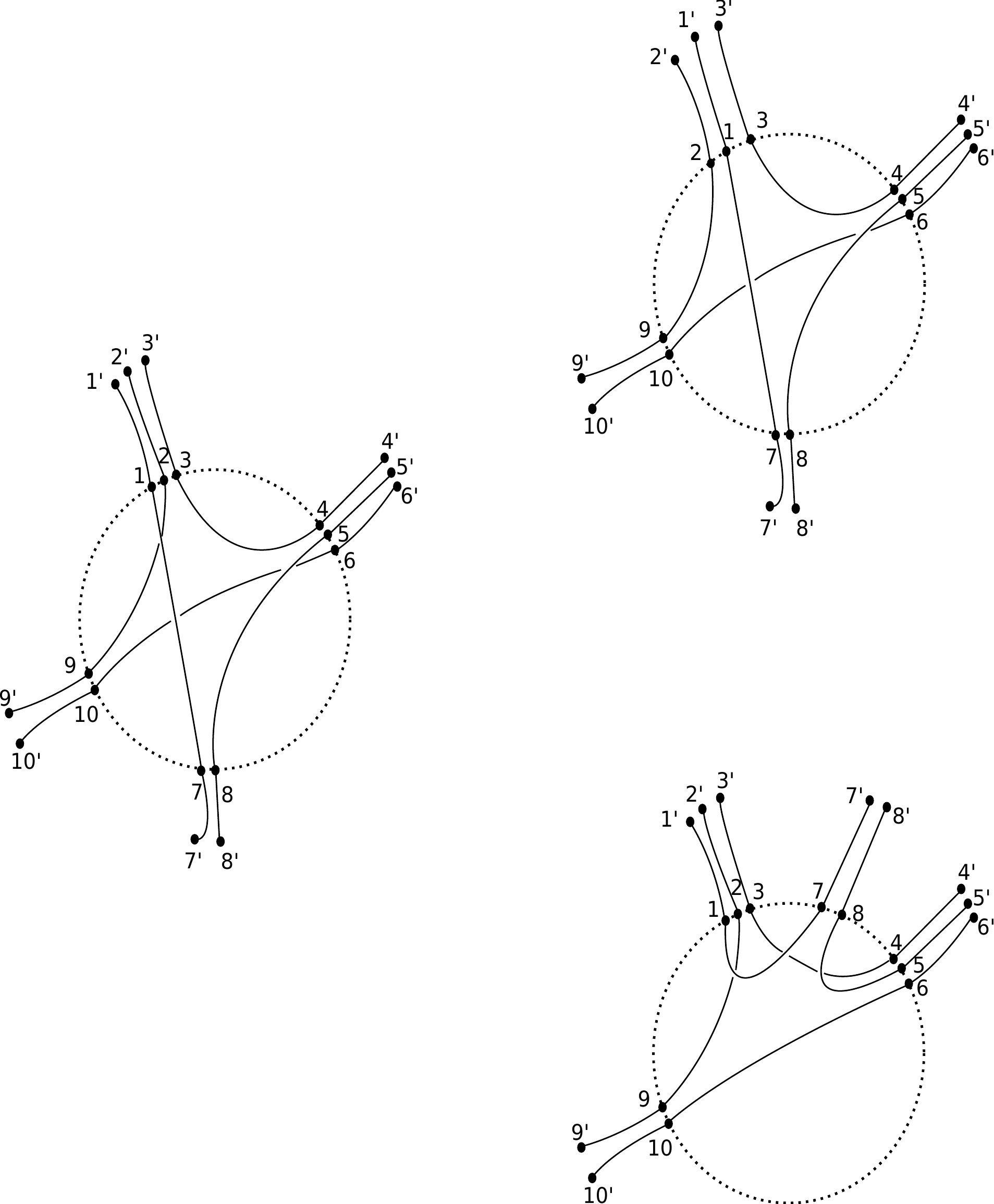}
\put(-220,42){$v_0$}
\put(-205,162){$e_1$}
\put(-130,135){$e_2$}
\put(-183,50){$e_3$}
\put(-240,75){$e_4$}
\put(-10,120){$v_1$}
\put(-85,210){$\tilde{e}_1$}
\put(2,190){$e_2$}
\put(-50,108){$e_3$}
\put(-108,135){$e_4$}
\put(-10,5){$v_2$}
\put(-85,75){$e_1$}
\put(2,60){$e_2$}
\put(-25,78){$e_3$}
\put(-108,5){$e_4$}

\caption{Three vertices $v_0$, $v_1$ and $v_2$ and their incident stranded edges
corresponding to equivalent conf\/igurations of stranded graphs, if the rest of each
graph is kept f\/ixed.}\label{fig:samevert}
\end{figure}
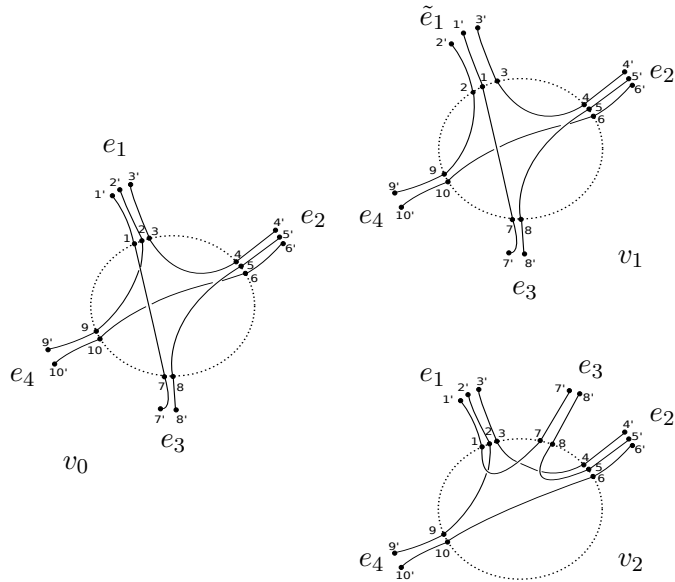

\looseness=-1
Consider a representative $\mathcal{G}_{\mathfrak{f}^0}$ of a rank 3 w-colored graph $\mathfrak{G}(\mathcal{V},\mathcal{E},\mathfrak{f}^0)$. We denote by $V( \mathcal{G}_{\mathfrak{f}^0})$, $E(\mathcal{G}_{\mathfrak{f}^0})$, $k( \mathcal{G}_{\mathfrak{f}^0})$, $f(\mathcal{G}_{\mathfrak{f}^0})$, $C_\partial( \mathcal{G}_{\mathfrak{f}^0})$, $E_\partial( \mathcal{G}_{\mathfrak{f}^0})$ and $F_\partial( \mathcal{G}_{\mathfrak{f}^0})$, the number of vertices, edges, connected components, half-edges, connected components of the boundary graphs, edges of the boundary graph and faces of the boundary graph of $\mathcal{G}_{\mathfrak{f}^0}$. $\rk( \mathcal{G}_{\mathfrak{f}^0})=V( \mathcal{G}_{\mathfrak{f}^0}) - k(\mathcal{G}_{\mathfrak{f}^0})$, $n( \mathcal{G}_{\mathfrak{f}^0})=E( \mathcal{G}_{\mathfrak{f}^0})-\rk( \mathcal{G}_{\mathfrak{f}^0})$ the rank and nullity of $\mathcal{G}_{\mathfrak{f}^0}$ and $B_{{\rm int}}(\mathcal{G}_{\mathfrak{f}^0})$, $B_{{\rm ext}}( \mathcal{G}_{\mathfrak{f}^0})$, the numbers of internal and external bubbles in~$\mathcal{G}_{\mathfrak{f}^0}$.

We now address the invariant on rank 3 w-colored graphs based on the sum of the genera of bubbles.

\begin{Proposition}[\cite{remia}]
\label{prop:euler3}
Let $\mathcal{G}_{\mathfrak{f}^0}$ be any representative of $\mathfrak{G}$ a rank~$3$ w-colored
graph. Then
\begin{gather*}
\gamma_3( \mathcal{G}_{\mathfrak{f}^0})= 3(V(\mathcal{G}_{\mathfrak{f}^0})-E(\mathcal{G}_{\mathfrak{f}^0})) + 2[F_{{\rm int}}( \mathcal{G}_{\mathfrak{f}^0}) - B_{{\rm int}}( \mathcal{G}_{\mathfrak{f}^0}) - B_{{\rm ext}}( \mathcal{G}_{\mathfrak{f}^0})] \leq 0.
\end{gather*}
\end{Proposition}

\begin{Definition}[topological invariant for rank 3 w-colored
graph]
Let $\mathfrak{G}(\mathcal{V},\mathcal{E},\mathfrak{f}^0)$ be a rank~3 w-colored graph.
The generalized topological invariant associated with $\mathfrak{G}$
is given by the fol\-lo\-wing function associated with any of its
representatives $\mathcal{G}_{\mathfrak{f}^0}$ (using the above notations)
\begin{gather}
\mathfrak{T}_{\mathfrak{G}}(x,y,z,s,w,q,t) =\mathfrak{T}_{ \mathcal{G}_{\mathfrak{f}^0}}(x,y,z,s,w,q,t)\nonumber\\
 \qquad{} = \sum_{A \Subset\mathcal{G}_{\mathfrak{f}^0}}
 (x-1)^{\rk( \mathcal{G}_{\mathfrak{f}^0})-\rk(A)}(y-1)^{n(A)}
z^{5k(A)-\gamma_3(A)}s^{C_\partial(A)} w^{F_{\partial}(A)} q^{E_{\partial}(A)} t^{f(A)},\label{ttopfla}
\end{gather}
with $\gamma_3(A)= 3(V(A)-E(A)) + 2[F_{{\rm int}}(A) - B_{{\rm int}}(A) - B_{{\rm ext}}(A)]$.
\end{Definition}

It is important to show that considering dif\/ferent representatives $\mathcal{G}_{\mathfrak{f}^0(\mathcal{G})}$ and $\mathcal{G}'_{\mathfrak{f}^0(\mathcal{G}')}$ of $\mathfrak{G}$, the def\/inition does not depend of the choice of the representative. In fact, there exists a one-to-one map between spanning c-subgraphs of $\mathcal{G}_{\mathfrak{f}^0(\mathcal{G})}$ and those of $\mathcal{G}'_{\mathfrak{f}^0(\mathcal{G}')}$. We can then map each $A \Subset \mathcal{G}_{\mathfrak{f}^0(\mathcal{G})}$ onto $A' \Subset \mathcal{G}'_{\mathfrak{f}^0(\mathcal{G}')}$ such that $A' = (A{\setminus} D_{ \mathcal{G}_{\mathfrak{f}^0(\mathcal{G})}}) \cup D_{\mathcal{G}'_{\mathfrak{f}^0(\mathcal{G}')}}$. We now verify that $\rk(\mathcal{G}_{\mathfrak{f}^0(\mathcal{G})})$, $\rk(A)$, $n(A)$ and $5k(A)-(3V+2F_{\rm int}(A))$
are independent of the representative. The exponents $B_{\rm int}(A)$, $B_{\rm ext}(A)$, $C_\partial (A)$,
$f(A)$, $E_\partial(A)$ and $F_\partial(A)$ only depend on $A{\setminus} D_{ \mathcal{G}_{\mathfrak{f}^0(\mathcal{G})}} = A'{\setminus} D_{ \mathcal{G}'_{\mathfrak{f}^0(\mathcal{G}')}}$ therefore are not dependent on the representative.

\begin{Proposition}[polynomial invariant]
$\mathfrak{T}_{\mathfrak{G}}=\mathfrak{T}_{\mathcal{G}_{\mathfrak{f}^0}}$ is a polynomial.
\end{Proposition}

 The main properties satisf\/ied by the above polynomial are given by the next statement.

\begin{Theorem}[contraction/cut rule for rank 3 w-colored graphs]
\label{theo:contens3}
Let $\mathfrak{G}$ be a rank $3$ w-colored graph. Then, for a regular edge~$e$ of any of the representative $ \mathcal{G}_{\mathfrak{f}^0}$ of $\mathfrak{G}$, we have
\begin{gather*}
\mathfrak{T}_{\mathcal{G}_{\mathfrak{f}^0}}=\mathfrak{T}_{ \mathcal{G}_{\mathfrak{f}^0}\vee e} +\mathfrak{T}_{ \mathcal{G}_{\mathfrak{f}^0}/e}.
\end{gather*}
For a bridge $e$, we have
\begin{gather*}
\mathfrak{T}_{ \mathcal{G}_{\mathfrak{f}^0} \vee e}= z^8s(wq)^3t^2 \mathfrak{T}_{ \mathcal{G}_{\mathfrak{f}^0}/e} \qquad \text{and}\qquad
\mathfrak{T}_{\mathcal{G}_{\mathfrak{f}^0}} =\big[(x-1)z^8s(wq)^3t^2+1\big]\mathfrak{T}_{ \mathcal{G}_{\mathfrak{f}^0}/e}.
\end{gather*}
For a trivial p-inner loop $e$, $p=0,1,2$, we have
\begin{gather*}
\mathfrak{T}_{\mathcal{G}_{\mathfrak{f}^0}}=\mathfrak{T}_{ \mathcal{G}_{\mathfrak{f}^0} \vee e} + (y-1)z^{4p-7} \mathfrak{T}_{ \mathcal{G}_{\mathfrak{f}^0}/e}.
\end{gather*}
\end{Theorem}
The proof of this statement uses bijections between c-subgraphs and the variations of the number of $p$-bubbles under contraction/cut rules.

There are several possible reductions of the above polynomial satisfying nice properties. Some of them are determined by a change of variables. We have
\begin{gather*}
\mathfrak{T}_{\mathcal{G}_{\mathfrak{f}^0}} \big(x,y,z,z^{-2},w,q,t\big) = \mathfrak{T}'_{\mathcal{G}_{\mathfrak{f}^0}} (x,y,z,w,q,t), \\
\mathfrak{T}_{\mathcal{G}_{\mathfrak{f}^0}} \big(x,y,z,z^{-2}s^{2},s^{-1},s,s^{-1}\big) = \mathfrak{T}''_{\mathcal{G}_{\mathfrak{f}^0}} (x,y,z,s),\\ 
\mathfrak{T}_{ \mathcal{G}_{\mathfrak{f}^0}} \big(x,y,z,z^{2}z^{-2},z^{-1},z,z^{-1}\big) = \mathfrak{T}'''_{\mathcal{G}_{\mathfrak{f}^0}} (x,y,z).
\end{gather*}
$\mathfrak{T}'_{ \mathcal{G}_{\mathfrak{f}^0}}$, $\mathfrak{T}''_{\mathcal{G}_{\mathfrak{f}^0}}$ and $\mathfrak{T}'''_{\mathcal{G}_{\mathfrak{f}^0}}$ satisfy also the contraction/cut rule. Furthermore $\mathfrak{T}_{\mathcal{G}_{\mathfrak{f}^0}}$ maps to Tutte polynomial by putting variables $z=w=q=t=1$. The mapping of $\mathfrak{T}_{\mathcal{G}_{\mathfrak{f}^0}}$ to BR polynomial is also possible but we must pay attention to the exponent of the variable $z$ which is modif\/ied by the number of additional discs. We will come back on this point in the following.

\begin{Definition}[multivariate form]
The multivariate form associated with \eqref{ttopfla} is def\/ined by
 \begin{gather*}
\widetilde{\mathfrak{T}}_{\mathfrak{G}}(x,\{\beta_e\},\{z_i\}_{i=1,2,3},s,w,q,t)=
\widetilde{\mathfrak{T}}_{\mathcal{G}_{\mathfrak{f}^0}}(x,\{\beta_e\},\{z_i\}_{i=1,2,3},s,w,q,t)\\ 
\qquad{} = \sum_{A \Subset \mathcal{G}_{\mathfrak{f}^0}}
 x^{\rk(A)}\left(\prod_{e\in A}\beta_e\right)
z_1^{F_{{\rm int}}(A)}z_2^{B_{{\rm int}}(A)} z_3^{B_{{\rm ext}}(A)}
s^{C_\partial(A)} w^{F_{\partial}(A)} q^{E_{\partial}(A)} t^{f(A)},
\end{gather*}
for $\{\beta_e\}_{e\in \mathcal{E}}$ labeling the edges of the graph $ \mathcal{G}_{\mathfrak{f}^0}$.
\end{Definition}
For all non-loop edge $e$, $\mathfrak{T}$ obeys the rule
$
 \widetilde{\mathfrak{T}}_{\mathfrak{G}}=
\widetilde{\mathfrak{T}}_{ \mathfrak{G}\vee e} +x\beta_e \widetilde{\mathfrak{T}}_{\mathfrak{G}/e}$,
which can be shown using standard techniques.

\section[The polynomial invariant for $4D$ w-colored graphs]{The polynomial invariant for $\boldsymbol{4D}$ w-colored graphs}
\label{subset:polymt}
Rank 4 w-colored graph structures are richer than the above case. This is completely expected since rank 4 w-colored graphs represent topologically $4D$ simplicial manifolds. For the present study, we observe that increasing the rank of the graph implies that the number of $p$-bubbles increases. As a consequence, we have a lot more freedom to def\/ine an invariant. In this work, we have not fully harnessed the number of possibilities of def\/ining invariants. This will be addressed in a forthcoming work.

Let us focus on the importance of the bubble combinatorics in the resolution of the contraction/cut invariant problem.
For the determination of the invariant in Proposition~\ref{prop:euler3}, each rank 3 w-colored graph is looked as collection of ribbon graphs. The invariant comes from the summation on the genera of all these graphs. The case of the rank 4 w-colored stranded graphs leads to several options. A~rank~4 w-colored stranded graph may be looked as a collection of $3$-bubbles which are ribbon graphs, a collection of $4$-bubbles which are themselves rank~3 graphs or a~mixture of both (see Fig.~\ref{fig:npb}). A sum over invariants of all 4-bubbles is introduced in Appendix~\ref{appb} for an interested reader. In general, we may also perform a summation of the invariants of all $3$-bubbles, the invariant of all $4$-bubbles or may mix both to f\/ind a new invariant.
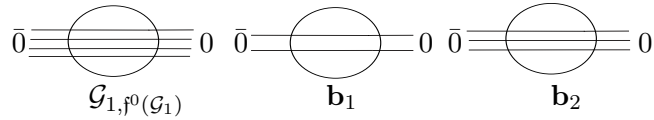
\begin{figure}[h]
 \centering

\includegraphics[angle=0, width=8cm, height=1cm]{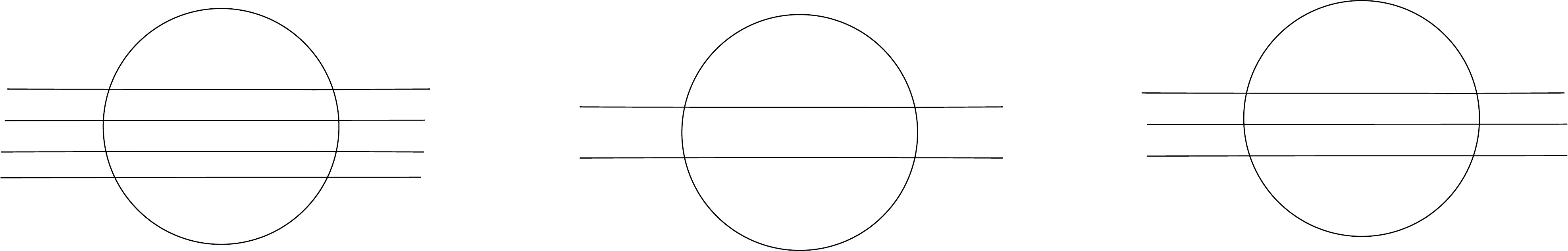}
\put(-205,-10){$ \mathcal{G}_{1,\mathfrak{f}^0(\mathcal{G}_1)}$}
\put(-115,-10){${\mathbf{b}}_1$}
\put(-163,10){$0$}
\put(-234,10){$\bar 0$}
\put(-80,10){$0$}
\put(-150,10){$\bar 0$}
\put(-30,-10){${\mathbf{b}}_2$}
\put(3,10){$0$}
\put(-69,10){$\bar 0$}

\caption{A 3-bubble ${\mathbf{b}}_1$ and a 4-bubble ${\mathbf{b}}_2$ of $ \mathcal{G}_{1,\mathfrak{f}^0(\mathcal{G}_1)}$.}\label{fig:npb}
\end{figure}

In the present work, we mainly discuss the summation of the invariants of all $3$-bubbles. This is to make contact with very recent results \cite{Gurau:2009tz, Ryan, Tanasa:2010me}.

Some notations deserve to be clarif\/ied. Consider a representative $ \mathcal{G}_{\mathfrak{f}^0}$ of any rank 4 w-colored graph: a $d$-bubble (closed or open) in $\mathcal{G}_{\mathfrak{f}^0}$ is denoted by ${\mathbf{b}}^d$, the set of $d$-bubbles is $\mathcal{B}^d$, and its cardinal $B^d$. $\mathcal{V}_{{\mathbf{b}}^d}$, $\mathcal{E}_{{\mathbf{b}}^d}$, $\mathcal{F}_{{\rm int}; {\mathbf{b}}^d}$ and $\mathcal{B}^p_{{\mathbf{b}}^d}$ are respectively the set of vertices, edges, internal faces and $p$-bubbles ($p \leq d$) of ${\mathbf{b}}^d$ of cardinal $V_{{\mathbf{b}}^d}$, $E_{{\mathbf{b}}^d}$, $F_{{\rm int}; {\mathbf{b}}^d}$ and $B^p_{{\mathbf{b}}^d}$ respectively.

\begin{Lemma}[rank 4 trivial loop contraction]\label{lem:tris}
Let $\mathfrak{G}$ a w-colored graph and $\mathcal{G}_{\mathfrak{f}^0}$ any of its representative
with boundary $\partial\mathcal{G}_{\mathfrak{f}^0}$,
 $e$ be a trivial loop of $ \mathcal{G}_{\mathfrak{f}^0}$, and $ \mathcal{G}'_{\mathfrak{f}^0}= \mathcal{G}_{\mathfrak{f}^0}/e$
the result of the contraction of $\mathcal{G}_{\mathfrak{f}^0}$ by $e$ with boundary
denoted by $\partial \mathcal{G}'_{\mathfrak{f}^0}$.
 Let $k$ denote the number of connected components
of $\mathcal{G}_{\mathfrak{f}^0}$, $V$ its number of vertices, $E$ its number of edges, $F_{{\rm int}}$ its number of faces, $B^i_{\rm int}$ its number of closed $i$-bubbles and $B^i_{\rm ext}$ its number of open $i$-bubbles $(i=3,4)$, $C_\partial$ the number of connected component of $\partial\mathcal{G}_{\mathfrak{f}^0}$, $f=V_\partial$ its number of stranded half-edges, $E_\partial=F_{{\rm ext}}$ the number of edges and $F_\partial$
the number of faces of $\partial \mathcal{G}_{\mathfrak{f}^0}$, and let $k'$, $V'$, $E'$, $F'_{{\rm int}}$, $C'_\partial$, $B'^i_{{\rm int}}$, $B'^i_{\rm ext} (i=3,4)$, $C'_\partial$, $f'$, $E'_\partial$, and $F'_\partial$ denote the similar numbers for $ \mathcal{G}'_{\mathfrak{f}^0}$ and its boundary $\partial \mathcal{G}'_{\mathfrak{f}^0}$.

If $e$ is a $4$-inner loop, then
\begin{gather}
k' = k + 3, \qquad
V' = V+3, \qquad
E' = E-1, \qquad
F'_{{\rm int}} = F_{{\rm int}}, \nonumber\\
C'_\partial= C_\partial, \qquad
f' = f, \qquad
E'_\partial = E_\partial,
\qquad F'_{\partial } = F_{\partial }, \nonumber\\
 B'^i_{\rm int} = B^i_{{\rm int}} -\complement_{4}^{i-1}, \qquad
B'^i_{\rm ext} = B^i_{\rm ext}, \qquad i=3,4.
\label{eq:cself2}
\end{gather}

If $e$ is a $3$-inner loop, then
\begin{gather}
k' = k + 3, \qquad
V' = V+3, \qquad
E' = E-1, \qquad
F'_{{\rm int}} = F_{{\rm int}}, \nonumber\\
C'_\partial= C_\partial, \qquad
f' = f, \qquad
E'_\partial = E_\partial,
\qquad F'_{\partial } = F_{\partial }, \nonumber\\
 B'^i_{\rm int} = B^i_{{\rm int}} -\complement_{3}^{i-1}, \qquad
B'^i_{\rm ext} = B^i_{\rm ext}, \qquad i=3,4.
\label{eq:cself3}
\end{gather}

If $e$ is a trivial $p$-inner loop such that $p=0,1,2$,
then
\begin{gather}
k' = k + 3, \qquad
V' = V+3, \qquad
E' = E-1, \qquad
F'_{{\rm int}} = F_{{\rm int}}, \nonumber\\
C'_\partial= C_\partial, \qquad
f' = f, \qquad
E'_\partial = E_\partial, \qquad
 F'_\partial = F_\partial, \nonumber\\
B'^3_{\rm int} + B'^3_{\rm ext} = B^3_{{\rm int}} + B^3_{\rm ext} + \beta_p, \qquad B'^4_{\rm int} + B'^4_{\rm ext} = B^4_{{\rm int}} + B^4_{\rm ext} + \beta'_p,
\label{eq:cselfp}
\end{gather}
where $\beta_{p} =6-3p$, $\beta'_{p} =8-3p$ and $\complement_{n}^{p}=\frac{n!}{p!(n-p)!}$ the ordinary binomial coefficient.
\end{Lemma}

\begin{proof}
The case of a 4-inner loop is that of a particular
closed graph constituted by a single vertex with a unique loop.
The contraction destroys the vertex, the six closed 3-bubbles, the four closed 4-bubbles and produces 4 discs. The result is immediate.

For $p$-inner loops, $p\leq 3$, there are other stranded half-edges
or edges on the same end vertex.
The two f\/irst lines of \eqref{eq:cself3}--\eqref{eq:cselfp} are direct using the def\/inition of the contraction that preserves (open and closed) strands.
We now turn on the number of bubbles of the graph.

Let us recall that we obtain the 3-bubbles by removing simultaneously from the initial graph two colors. The 4-bubbles are obtained by removing all strands with the pair of colors. In all the following discussion, removing a strand means removing all others having the same color.

Let us focus on the 3-inner loop case. We have 3 strands which are closed and one outer strand. By contracting this edge, we loose the 3 3-bubbles passing through the inner strands. The 3-bubbles which use the outer strand are still present after the contraction and just loose some internal faces. We also loose the only one closed 4-bubble passing through these inner strands. All 4-bubbles using the outer strand are still present after the contraction. This ends the proof of~\eqref{eq:cself3}.

In the case of the 2-inner loop, we have two outer strands and two strands which are immediately closed (the inner ones). Contracting the edge, we loose the closed 3-bubble formed by the two inner strands obtained by removing the two outer strands. Any other 3-bubble is obtained either by removing one of the outer strands and one of inner or by removing the two inner strands. The number of 3-bubbles (closed or opened) obtained in $\mathcal{G}_{\mathfrak{f}^0}$ or $\mathcal{G}'_{\mathfrak{f}^0}$ are the same performing the f\/irst operation. If we remove the two inner strands, the 3-bubbles using the outer strands in $\mathcal{G}_{\mathfrak{f}^0}$ get split in two parts in $\mathcal{G}'_{\mathfrak{f}^0}$. We then obtain one more bubble in $ \mathcal{G}'_{\mathfrak{f}^0}$. This compensates the one we have lost previously. Hence $B'^3_{\rm int} + B'^3_{\rm ext} = B^3_{{\rm int}} + B^3_{\rm ext}$.

The number of 4-bubbles obtained by removing one of the outer strands is preserved in the contracted graph.
If we remove one of the inner strands, we obtain one extra 4-bubble in the contracted graph. Then $B'^4_{\rm int} + B'^4_{\rm ext} = B^4_{{\rm int}} + B^4_{\rm ext}+2$.

For the 1-inner loop, the procedure is similar. We have three outer strands and one inner. Removing the inner and one of the outer strands, we obtain one more 3-bubble in the contracted graph but if we remove two of the outers, the number of 3-bubbles in the contracted graph is the same. We have $B'^3_{\rm int} + B'^3_{\rm ext} = B^3_{{\rm int}} + B^3_{\rm ext}+3$.
	
	We now concentrate on the $4$-bubbles case. If we remove the inner strand, we obtain two more $4$-bubbles in the contracted graph. Otherwise, if we remove one of the outer strands, we obtain one more $4$-bubbles in the contracted graph. Hence $B'^4_{\rm int} + B'^4_{\rm ext} = B^4_{{\rm int}} + B^4_{\rm ext}+5$.

Case of the 0-inner loop. By implementing the above techniques we can prove also $B'^3_{\rm int} + B'^3_{\rm ext} = B^3_{{\rm int}} + B^3_{\rm ext}+6$ and $B'^4_{\rm int} + B'^4_{\rm ext} = B^4_{{\rm int}} + B^4_{\rm ext}+8$.
\end{proof}

\begin{Lemma}[cut/contraction of special edges]\label{lem:cutbri}
Let $ \mathcal{G}_{\mathfrak{f}^0}$ be a representative of $\mathfrak{G}$ a~rank~$4$ w-colored graph
and $e$ an edge in $\mathcal{G}_{\mathfrak{f}^0}$. 

If $e$ is a bridge, we have
\begin{gather}
k( \mathcal{G}_{\mathfrak{f}^0}\vee e) = k(\mathcal{G}_{\mathfrak{f}^0}/e)+1,\qquad
V(\mathcal{G}_{\mathfrak{f}^0}\vee e) = V( \mathcal{G}_{\mathfrak{f}^0}/e) + 1,\nonumber\\
E( \mathcal{G}_{\mathfrak{f}^0}\vee e) = E( \mathcal{G}_{\mathfrak{f}^0}/e),\qquad
f( \mathcal{G}_{\mathfrak{f}^0} \vee e)= f(\mathcal{G}_{\mathfrak{f}^0}/e) +2, \label{eq:kve} \\
F_{{\rm int}}( \mathcal{G}_{\mathfrak{f}^0}\vee e)= F_{{\rm int}}(\mathcal{G}_{\mathfrak{f}^0}/e), \qquad
B^3_{{\rm int}}( \mathcal{G}_{\mathfrak{f}^0} \vee e)=B^3_{{\rm int}}( \mathcal{G}_{\mathfrak{f}^0}/e), \nonumber\\
B^4_{{\rm int}}( \mathcal{G}_{\mathfrak{f}^0} \vee e)=B^4_{{\rm int}}(\mathcal{G}_{\mathfrak{f}^0}/e),
\label{eq:fint} \\
C_\partial(\mathcal{G}_{\mathfrak{f}^0} \vee e) = C_\partial( \mathcal{G}_{\mathfrak{f}^0}/e) +1,
\qquad
E_\partial(\mathcal{G}_{\mathfrak{f}^0} \vee e) = E_\partial ( \mathcal{G}_{\mathfrak{f}^0}/e) +4, \nonumber\\
F_\partial(\mathcal{G}_{\mathfrak{f}^0} \vee e) = F_\partial( \mathcal{G}_{\mathfrak{f}^0}/e) + 6,
\label{eq:cextbord} \\
 B^3_{{\rm ext}}( \mathcal{G}_{\mathfrak{f}^0}\vee e) = B^3_{{\rm ext}}( \mathcal{G}_{\mathfrak{f}^0}/e) + 6,\qquad
 B^4_{{\rm ext}}( \mathcal{G}_{\mathfrak{f}^0}\vee e) = B^4_{{\rm ext}}(\mathcal{G}_{\mathfrak{f}^0}/e) + 4.
\label{eq:bubext}
\end{gather}

If $e$ is a trivial $p$-inner loop, $p=0,1,2,3$, we have
\begin{gather}
k( \mathcal{G}_{\mathfrak{f}^0}\vee e) = k(\mathcal{G}_{\mathfrak{f}^0}/e)-3,\qquad
V(\mathcal{G}_{\mathfrak{f}^0}\vee e) = V( \mathcal{G}_{\mathfrak{f}^0}/e)- 3,\qquad
E( \mathcal{G}_{\mathfrak{f}^0}\vee e) = E(\mathcal{G}_{\mathfrak{f}^0}/e),\nonumber\\
f( \mathcal{G}_{\mathfrak{f}^0} \vee e)= f(\mathcal{G}_{\mathfrak{f}^0}/e) +2, \label{eq:kvep} \\
F_{{\rm int}}( \mathcal{G}_{\mathfrak{f}^0}\vee e)+ C_\partial( \mathcal{G}_{\mathfrak{f}^0} \vee e)
= F_{{\rm int}}( \mathcal{G}_{\mathfrak{f}^0}/e) + C_\partial(\mathcal{G}_{\mathfrak{f}^0}/e)-3,\nonumber\\
 E_\partial(\mathcal{G}_{\mathfrak{f}^0} \vee e) = E_\partial (\mathcal{G}_{\mathfrak{f}^0}/e) + 4,\label{eq:fintp} \\
B^3_{{\rm int}}( \mathcal{G}_{\mathfrak{f}^0} \vee e) + B^3_{{\rm ext}}( \mathcal{G}_{\mathfrak{f}^0}\vee e)
=B^3_{{\rm int}}( \mathcal{G}_{\mathfrak{f}^0}/e) + B^3_{{\rm ext}}( \mathcal{G}_{\mathfrak{f}^0}/e) - (6-3p),\label{eq:cextp} \\
B^4_{{\rm int}}( \mathcal{G}_{\mathfrak{f}^0} \vee e) + B^4_{{\rm ext}}(\mathcal{G}_{\mathfrak{f}^0}\vee e)
=B^4_{{\rm int}}( \mathcal{G}_{\mathfrak{f}^0}/e) + B^4_{{\rm ext}}(\mathcal{G}_{\mathfrak{f}^0}/e) - (8-3p).
\label{eq:cextp4}
\end{gather}
\end{Lemma}

\begin{proof} Let us begin by the case of the bridge.
In~\eqref{eq:kve}, the relations are directly achieved.
Now, we turn on \eqref{eq:fint}. From \cite[Lemma~7]{remia}, we claim that the faces passing through~$e$ are open and belong to the same connected component of the boundary graph. Then
all closed faces on each side of the bridge are preserved
after cutting~$e$. The same are still preserved after
edge contraction and therefore $F_{{\rm int}}(\mathcal{G}_{\mathfrak{f}^0}\vee e) = F_{{\rm int}}( \mathcal{G}_{\mathfrak{f}^0}/e)$,
$B^3_{{\rm int}}( \mathcal{G}_{\mathfrak{f}^0} \vee e)=B^3_{{\rm int}}( \mathcal{G}_{\mathfrak{f}^0}/e)$ and $B^4_{{\rm int}}( \mathcal{G}_{\mathfrak{f}^0} \vee e)=B^4_{{\rm int}}( \mathcal{G}_{\mathfrak{f}^0}/e)$.
Let us focus on~\eqref{eq:cextbord}. Since the faces passing through $e$ belong to the same connected component, cutting~$e$, this unique component
yields two boundary components. We obtain
$C_\partial( \mathcal{G}_{\mathfrak{f}^0} \vee e) = C_\partial( \mathcal{G}_{\mathfrak{f}^0}/e) +1$,
$E_\partial( \mathcal{G}_{\mathfrak{f}^0} \vee e)= E_\partial (\mathcal{G}_{\mathfrak{f}^0}/e)+4$
(the cut of~$e$ divides each external face into two dif\/ferent
open strands) and
$F_\partial( \mathcal{G}_{\mathfrak{f}^0} \vee e)= F_\partial (\mathcal{G}_{\mathfrak{f}^0}/e)+6$
since $C_\partial( \mathcal{G}_{\mathfrak{f}^0}/e)=C_\partial(\mathcal{G}_{\mathfrak{f}^0})$,
$E_\partial ( \mathcal{G}_{\mathfrak{f}^0}/e)=E_\partial (\mathcal{G}_{\mathfrak{f}^0})$
and $F_\partial( \mathcal{G}_{\mathfrak{f}^0} /e)= F_\partial (\mathcal{G}_{\mathfrak{f}^0})$
 which are direct from Lemma~\ref{lem:fullcont}.
For the number of external bubbles, there are six 3-bubbles and four 4-bubbles in~$ \mathcal{G}_{\mathfrak{f}^0}$ passing through the bridge. These bubbles are present in $ \mathcal{G}_{\mathfrak{f}^0}/e$ and cutting the bridge
each of these bubbles splits in two. This yields~\eqref{eq:bubext}.

We concentrate now on a trivial $p$-inner loop $e$.
The relations~\eqref{eq:kvep} is directly found.
We focus on the rest of the equations.
Consider the faces $f_i$ in $e$ made with outer strands.
For $p=0,1,2,3$, we have $f_{i}$, $1\leq i\leq 4-p$. These
faces can be open or closed. We do a case by case study
according to the number of open or closed faces among the
$f_i$'s.

Suppose that $4-p$ of $f_i$'s are closed.
Cutting $e$ entails
\begin{gather*}
F_{{\rm int}}( \mathcal{G}_{\mathfrak{f}^0}\vee e) = F_{{\rm int}} (\mathcal{G}_{\mathfrak{f}^0}) -4
,\qquad\!
C_{\partial}( \mathcal{G}_{\mathfrak{f}^0} \vee e) = C_\partial ( \mathcal{G}_{\mathfrak{f}^0}) +1
, \qquad\!
 E_\partial( \mathcal{G}_{\mathfrak{f}^0} \vee e) = E_\partial(\mathcal{G}_{\mathfrak{f}^0}) +4,\\
B^3(\mathcal{G}_{\mathfrak{f}^0} \vee e) = B^3(\mathcal{G}_{\mathfrak{f}^0}),\qquad B^4( \mathcal{G}_{\mathfrak{f}^0} \vee e) = B^4( \mathcal{G}_{\mathfrak{f}^0}),\qquad
 F_\partial(\mathcal{G}_{\mathfrak{f}^0} \vee e) = F_\partial( \mathcal{G}_{\mathfrak{f}^0}) + 6.
\end{gather*}
Remark that, in this situation, only the variation of the total number of bubbles can be known.

Suppose that $4-p-1$ of the $f_i$'s are closed
and one is open. Cutting $e$ implies
\begin{gather*}
F_{{\rm int}}( \mathcal{G}_{\mathfrak{f}^0}\vee e) = F_{{\rm int}} (\mathcal{G}_{\mathfrak{f}^0}) -3
,\qquad
C_{\partial}( \mathcal{G}_{\mathfrak{f}^0} \vee e) = C_\partial ( \mathcal{G}_{\mathfrak{f}^0})
, \qquad
 E_\partial( \mathcal{G}_{\mathfrak{f}^0} \vee e) = E_\partial(\mathcal{G}_{\mathfrak{f}^0}) +4,\\
B^3( \mathcal{G}_{\mathfrak{f}^0} \vee e) = B^3( \mathcal{G}_{\mathfrak{f}^0}),\qquad B^4( \mathcal{G}_{\mathfrak{f}^0} \vee e) = B^4( \mathcal{G}_{\mathfrak{f}^0}),
\qquad
 F_\partial( \mathcal{G}_{\mathfrak{f}^0} \vee e) = F_\partial( \mathcal{G}_{\mathfrak{f}^0}) +3.
\end{gather*}

Suppose that $4-p-2$ of the $f_i$'s are closed
and two are open. Cutting~$e$ gives
\begin{gather*}
F_{{\rm int}}( \mathcal{G}_{\mathfrak{f}^0}\vee e) = F_{{\rm int}} ( \mathcal{G}_{\mathfrak{f}^0}) -2
,\qquad\!
C_{\partial}( \mathcal{G}_{\mathfrak{f}^0} \vee e) = C_\partial (\mathcal{G}_{\mathfrak{f}^0}) -1
,\qquad\!
 E_\partial(\mathcal{G}_{\mathfrak{f}^0} \vee e) = E_\partial( \mathcal{G}_{\mathfrak{f}^0}) +4,\\
B^3( \mathcal{G}_{\mathfrak{f}^0} \vee e) = B^3( \mathcal{G}_{\mathfrak{f}^0}),\qquad B^4( \mathcal{G}_{\mathfrak{f}^0} \vee e) = B^4( \mathcal{G}_{\mathfrak{f}^0}),\qquad
 F_\partial( \mathcal{G}_{\mathfrak{f}^0} \vee e) = F_\partial(\mathcal{G}_{\mathfrak{f}^0}).
\end{gather*}

Suppose that $4-p-3$ of the $f_i$'s are closed
and three are open. Cutting $e$ gives
\begin{gather*}
F_{{\rm int}}( \mathcal{G}_{\mathfrak{f}^0}\vee e) = F_{{\rm int}} ( \mathcal{G}_{\mathfrak{f}^0}) -1
,\qquad\!
C_{\partial}( \mathcal{G}_{\mathfrak{f}^0} \vee e) = C_\partial ( \mathcal{G}_{\mathfrak{f}^0}) -2
,\qquad\!
 E_\partial( \mathcal{G}_{\mathfrak{f}^0} \vee e) = E_\partial( \mathcal{G}_{\mathfrak{f}^0}) +4,\\
B^3( \mathcal{G}_{\mathfrak{f}^0} \vee e) = B^3( \mathcal{G}_{\mathfrak{f}^0}),\qquad B^4( \mathcal{G}_{\mathfrak{f}^0} \vee e) = B^4(\mathcal{G}_{\mathfrak{f}^0}),\qquad
 F_\partial(\mathcal{G}_{\mathfrak{f}^0} \vee e) = F_\partial( \mathcal{G}_{\mathfrak{f}^0})-3.
\end{gather*}
Note that this case does not apply for $p=2$.

For $p=0$ an additional situation applies: assume that all
 four $f_i$'s are open. Cutting $e$ gives
\begin{gather*}
F_{{\rm int}}( \mathcal{G}_{\mathfrak{f}^0}\vee e) = F_{{\rm int}} (\mathcal{G}_{\mathfrak{f}^0})
,\qquad
C_{\partial}( \mathcal{G}_{\mathfrak{f}^0} \vee e) = C_\partial ( \mathcal{G}_{\mathfrak{f}^0}) -3
,\qquad
 E_\partial( \mathcal{G}_{\mathfrak{f}^0}\vee e) = E_\partial( \mathcal{G}_{\mathfrak{f}^0}) +4,\\
B^3( \mathcal{G}_{\mathfrak{f}^0} \vee e) = B^3( \mathcal{G}_{\mathfrak{f}^0}),\qquad B^4( \mathcal{G}_{\mathfrak{f}^0} \vee e) = B^4( \mathcal{G}_{\mathfrak{f}^0}),\qquad
 F_\partial(\mathcal{G}_{\mathfrak{f}^0} \vee e) = F_\partial( \mathcal{G}_{\mathfrak{f}^0})- 6.
\end{gather*}
Lemma~\ref{lem:tris} relates the same numbers for $ \mathcal{G}_{\mathfrak{f}^0}/e$ and $ \mathcal{G}_{\mathfrak{f}^0}$ from which one is able to prove~\eqref{eq:fintp},~\eqref{eq:cextp} and~\eqref{eq:cextp4}.
\end{proof}

Before focusing on the invariant, let us establish an intermediate result.

\begin{Lemma}\label{lemma:bulbes4}
For any representative $\mathcal{G}$ of a~rank~$4$ w-colored graph $\mathfrak{G}$, we have
\begin{gather} \label{eq:vefsum}
\sum_{{\mathbf{b}}^3 \in \mathcal{B}^3} V_{{\mathbf{b}}^3} \geq 6V, \qquad \sum_{{\mathbf{b}}^3 \in \mathcal{B}^3} E_{{\mathbf{b}}^3} = 6E, \qquad \sum_{{\mathbf{b}}^3\in \mathcal{B}^3} F_{{\rm int};{\mathbf{b}}^3} = 3F_{{\rm int}}.
\end{gather}
Moreover,
\begin{gather} \label{eq:sumdouble}
\sum_{{\mathbf{b}}^4 \in \mathcal{B}^4} B^3_{{\mathbf{b}}^4}=2B^3, \qquad 3B^4\leq2B^3, \qquad B^4\geq 4.
\end{gather}
\end{Lemma}

\begin{proof}
Each vertex of $\mathcal{G}_{\mathfrak{f}^0}$ can be decomposed at least in 6 vertices (6 vertices is the minimum given by the simplest vertex of the form~$ \mathcal{G}_{1,\mathfrak{f}^0(\mathcal{G}_1)}$ in Fig.~\ref{fig:npb} for any rank 4 w-colored graph) which could belong to a $3$-bubble. This gives the f\/irst relation in \eqref{eq:vefsum}. Furthermore, using the colors, one observes that each edge of~$\mathcal{G}_{\mathfrak{f}^0}$ splits into 6 ribbon edges, and each internal face of~$ \mathcal{G}_{\mathfrak{f}^0}$ belongs to three $3$-bubbles. Thus we have the rest of equations in~\eqref{eq:vefsum}.

We concentrate now on \eqref{eq:sumdouble}. Consider a $3$-bubble ${\mathbf{b}}^3$ in $\mathcal{G}_{\mathfrak{f}^0}$, ${\mathbf{b}}^3$ is made with~$3$ colors of~$\mathcal{G}_{\mathfrak{f}^0}$. Since $ \mathcal{G}_{\mathfrak{f}^0}$ is made with $5$ colors, ${\mathbf{b}}^3$ is contained exactly in two $4$-bubbles obtained by adding to the three colors in~${\mathbf{b}}^3$ one of the two remaining colors of~$\mathcal{G}_{\mathfrak{f}^0}$. Hence, the f\/irst equation in~\eqref{eq:sumdouble} follows. Using again the color prescription, each $4$-bubble has at least three $3$-bubbles, $B^3_{{\mathbf{b}}^4}\geq3$. Summing each member of this relation on the set of $4$-bubbles, we end the proof the second equation in~\eqref{eq:sumdouble}.
\end{proof}

Now, we can introduce a parametrized invariant on this category of graph.

\begin{Proposition}
\label{prop:euler4}
Let $ \mathcal{G}_{\mathfrak{f}^0}$ be any representative of a rank $4$ w-colored graph $\mathfrak{G}$. There exist some positive rational numbers $\alpha_3$ and $\alpha_4$ such that
\begin{gather*} 
\gamma_{4;(\alpha_3,\alpha_4)}( \mathcal{G}_{\mathfrak{f}^0})= 6(V-E) + 3F_{{\rm int}} - 2(1+\alpha_3)B^3+ (3\alpha_3-\alpha_4)B^4
\end{gather*}
is a negative integer.
\end{Proposition}
\begin{proof}
Consider a representative $\mathcal{G}_{\mathfrak{f}^0}$ of a rank $4$ w-colored graph. The $3$-bubbles of~$ \mathcal{G}_{\mathfrak{f}^0}$ are connected rank~$2$ stranded graphs or ribbon graphs with half-edges. Consider a $3$-bubble~${\mathbf{b}}^3$ of~$\mathcal{G}_{\mathfrak{f}^0}$ and~$\widetilde{{\mathbf{b}}^3}$ its underline ribbon graph. Using the Euler formula, one has
\begin{gather*} 
V_{{\mathbf{b}}^3}-E_{{\mathbf{b}}^3}+ F_{{\rm int};{\mathbf{b}}^3}+C_\partial\big({\mathbf{b}}^3\big)-2=-\kappa\leq 0,
\end{gather*}
with $\kappa$ the genus of $\widetilde{{\mathbf{b}}^3}$. Hence,
\begin{gather} \label{eq:forran2}
V_{{\mathbf{b}}^3}-E_{{\mathbf{b}}^3}+ F_{{\rm int};{\mathbf{b}}^3}\leq 2.
\end{gather}
A summation on all the $3$-bubbles of $ \mathcal{G}_{\mathfrak{f}^0}$ gives
\begin{gather*}
\sum_{{\mathbf{b}}^3 \in \mathcal{B}^3}(V_{{\mathbf{b}}^3}-E_{{\mathbf{b}}^3})+ \sum_{{\mathbf{b}}^3 \in \mathcal{B}^3}F_{{\rm int};{\mathbf{b}}^3} -2B^3\leq 0.
\end{gather*}
From Lemma \ref{lemma:bulbes4}, we have
\begin{gather} \label{eq:forrank4}
6(V-E) + 3F_{{\rm int}} - 2B^3\leq 0, \qquad 3B^4 - 2B^3\leq0, \qquad -B^4+4\leq 0.
\end{gather}

Consider some positive rational numbers $\alpha_3$ and $\alpha_4$ such that $- 2(1+\alpha_3)B^3+ (3\alpha_3-\alpha_4)B^4$ is an integer. Multiplying the two last relations in~\eqref{eq:forrank4} by~$\alpha_3$ and~$\alpha_4$, respectively, we obtain~$3\alpha_3B^4 - 2\alpha_3B^3\leq0$ and $-\alpha_4B^4+4\alpha_4\leq0$. Combining these two last relations with the f\/irst relation in~\eqref{eq:forrank4}, we obtain
\begin{gather} \label{eq:forran4}
6(V-E) + 3F_{{\rm int}} - 2(1+\alpha_3)B^3+ (3\alpha_3-\alpha_4)B^4+4\alpha_4\leq 0,
\end{gather}
i.e., $\gamma_{4;(\alpha_3,\alpha_4)}( \mathcal{G}_{\mathfrak{f}^0})$ is a negative integer.
\end{proof}

Clearly, the quantity $\gamma_{4;(\alpha_3,\alpha_4)}(\cdot)$ is a negative integer for all non negative integers $\alpha_3$ and $\alpha_4$. Furthermore, $\gamma_{4;(\alpha_3,\alpha_4)}(\cdot)$ is still a negative integer if $\alpha_3$ and $\alpha_4$ are equal to specif\/ic positive rational numbers.

\looseness=-1
We consider the combinatorial object $6(V-E) + 3F_{{\rm int}} - 2(1+\alpha_3)B^3+ (3\alpha_3-\alpha_4)B^4$ as a~ge\-ne\-ra\=li\-zed Euler characteristic. Indeed, f\/ixing $3\alpha_3-\alpha_4>0$ (since $1+\alpha_3>0$), this integer is reminiscent of a~weighted form of the Euler characteristic, if one considers the $B^d$'s as Betti numbers. Of course, one also expects an alternating sum of positive integers as in the Euler charac\-te\-ris\-tic. Here we see that we can achieve this only by tuning properly the parameters alphas.

We can def\/ine the polynomial invariant on rank 4 w-colored graphs.

\begin{Definition}[topological invariant for rank 4 w-colored
graph]\label{def:topotens4}
Let $\mathfrak{G}(\mathcal{V},\mathcal{E},\mathfrak{f}^0)$ be a rank~4 w-colored graph and~$\alpha_3$ and $\alpha_4$ some positive rational numbers.
The generalized topological invariant associated with~$\mathfrak{G}$
is given by the following function associated with any of its
representatives $\mathcal{G}$ (using the above notations)
\begin{gather*}
\mathfrak{T}_{\mathfrak{G};(\alpha_3,\alpha_4)}(x,y,z,s,w,q,t) =\mathfrak{T}_{ \mathcal{G}_{\mathfrak{f}^0};(\alpha_3,\alpha_4)}(x,y,z,s,w,q,t)\\ 
\qquad{} =\sum_{A \Subset \mathcal{G}_{\mathfrak{f}^0}}
 (x-1)^{\rk(\mathcal{G})-\rk(A)}(y-1)^{n(A)}
z^{9k(A)-\gamma_{4;(\alpha_3,\alpha_4)}(A)}s^{C_\partial(A)} w^{F_{\partial}(A)} q^{E_{\partial}(A)} t^{f(A)},
\end{gather*}
with $\gamma_{4;(\alpha_3,\alpha_4)}(A)= 6(V(A)-E(A)) + 3F_{{\rm int}}(A) - 2(1+\alpha_3)B^3(A)+ (3\alpha_3-\alpha_4)B^4(A)$ a negative integer.
\end{Definition}

Comparing this polynomial with the one def\/ined on rank 3 w-colored of c-subgraphs, we see that in addition to the 3-bubbles, the polynomial invariant also takes into consideration 4-bubbles.

We can now introduce our main theorem.

\begin{Theorem}[contraction/cut rule for w-colored graphs]
\label{theo:contens4}
Let $\mathfrak{G}$ be a rank~$4$ w-colored graph and~$\alpha_3$ and~$\alpha_4$ some positive rational numbers. Then, for a regular edge~$e$ of any of the representa\-ti\-ve~$ \mathcal{G}_{\mathfrak{f}^0}$ of $\mathfrak{G}$, we have
\begin{gather}
\mathfrak{T}_{ \mathcal{G}_{\mathfrak{f}^0};(\alpha_3,\alpha_4)}=\mathfrak{T}_{ \mathcal{G}_{\mathfrak{f}^0}\vee e;(\alpha_3,\alpha_4)} +\mathfrak{T}_{ \mathcal{G}_{\mathfrak{f}^0}/e;(\alpha_3,\alpha_4)}.
\label{tenscondel}
\end{gather}
For a bridge $e$, we have $\mathfrak{T}_{\mathcal{G}_{\mathfrak{f}^0} \vee e;(\alpha_3,\alpha_4)}= z^{15+4\alpha_4}sw^6q^4t^2 \mathfrak{T}_{ \mathcal{G}_{\mathfrak{f}^0}/e;(\alpha_3,\alpha_4)}$
and
\begin{gather}
\mathfrak{T}_{\mathcal{G}_{\mathfrak{f}^0};(\alpha_3,\alpha_4)} =[(x-1)z^{15+4\alpha_4}sw^6q^4t^2+1] \mathfrak{T}_{ \mathcal{G}_{\mathfrak{f}^0}/e;(\alpha_3,\alpha_4)}.
\label{tensbri}
\end{gather}
For a trivial p-inner loop $e$, $p=0,1,2,3$, we have
\begin{gather}
\mathfrak{T}_{ \mathcal{G}_{\mathfrak{f}^0};(\alpha_3,\alpha_4)}=\mathfrak{T}_{ \mathcal{G}_{\mathfrak{f}^0} \vee e;(\alpha_3,\alpha_4)} + (y-1)z^{-15+6p+3\alpha_3(4-p)+\alpha_4(3p-8)} \mathfrak{T}_{\mathcal{G}_{\mathfrak{f}^0}/e;(\alpha_3,\alpha_4)}.
\label{tensself}
\end{gather}
\end{Theorem}

\begin{proof}
Let $\mathcal{G}_{\mathfrak{f}^0}$ be a representative of a rank 4 w-colored graph $\mathfrak{G}$. Our main concern is the change in the number of internal and external $i$-bubbles ($i=3,4$) and the independence of the f\/inal
result on the representative $ \mathcal{G}_{\mathfrak{f}^0}$. From Proposition~\ref{lem:cutmg}, we can def\/ine $\mathfrak{G}\vee e$ and $\mathfrak{G}/e$ and conclude that the resulting polynomials $\mathfrak{T}_{\mathcal{G}_{\mathfrak{f}^0} \vee e;(\alpha_3,\alpha_4)}$ and $\mathfrak{T}_{\mathcal{G}_{\mathfrak{f}^0} /e;(\alpha_3,\alpha_4)}$ are independent on the representative $ \mathcal{G}_{\mathfrak{f}^0} \vee e$ and $ \mathcal{G}_{\mathfrak{f}^0}/e$, respectively.

Let us prove the equation \eqref{tenscondel}.
 Consider an ordinary edge $e$ of $ \mathcal{G}_{\mathfrak{f}^0}$, the set of spanning c-subgraphs which do not contain $e$ being the same as the set of spanning c-subgraphs of $ \mathcal{G}_{\mathfrak{f}^0} \vee e$, the number of open and closed $i$-bubbles ($i=3,4$) on each subgraph is the same, it is direct to get $ \sum\limits_{A \Subset \mathcal{G}_{\mathfrak{f}^0} ; e\notin A} (\cdot)=\mathfrak{T}_{ \mathcal{G}_{\mathfrak{f}^0} \vee e;(\alpha_3,\alpha_4)} $. The strands of $e$ are clearly preserved after the contraction.
The bubbles which do not pass through $e$ are not af\/fected at all
by the procedure. The bubbles passing through~$e$
are also preserved since the contraction
do not delete faces or strands. Hence
$\sum\limits_{A \Subset\mathcal{G}_{\mathfrak{f}^0} ; e\in A} (\cdot)=\mathfrak{T}_{ \mathcal{G}_{\mathfrak{f}^0} /e;(\alpha_3,\alpha_4)} $.

We now concentrate on the bridge case and \eqref{tensbri}.
Cutting a bridge yields, from the sum $\sum\limits_{A \Subset \mathcal{G}_{\mathfrak{f}^0} ; e\notin A}$, the product $(x-1)\mathfrak{T}_{ \mathcal{G}_{\mathfrak{f}^0} \vee e;(\alpha_3,\alpha_4)}$. Using the mapping between
$\{A \Subset \mathcal{G}_{\mathfrak{f}^0};e\in A\}$ and $\{A\Subset \mathcal{G}_{\mathfrak{f}^0}/e\}$, we obtain $\sum\limits_{A \Subset \mathcal{G}_{\mathfrak{f}^0} ; e\in A} (\cdot)=\mathfrak{T}_{ \mathcal{G}_{\mathfrak{f}^0} /e;(\alpha_3,\alpha_4)} $.
We now investigate the relation between $\mathfrak{T}_{ \mathcal{G}_{\mathfrak{f}^0} \vee e;(\alpha_3,\alpha_4)}$ and $\mathfrak{T}_{ \mathcal{G}_{\mathfrak{f}^0}/e;(\alpha_3,\alpha_4)}$.
There is a bijection between $A\Subset \mathcal{G}_{\mathfrak{f}^0} \vee e$ and
$A' \Subset \mathcal{G}_{\mathfrak{f}^0}/e$ where each $A$ and $A'$ are both
uniquely related to some $A_0\Subset \mathcal{G}_{\mathfrak{f}^0}$ as $A = A_0\vee e$
and $A'=A_0/e$. Using Lemma~\ref{lem:cutbri},
the relation~\eqref{tensbri} follows.

 Let us investigate the trivial $p$-inner loop case and prove~\eqref{tensself}.
As discussed earlier, the relation $\sum\limits_{A \Subset \mathcal{G}_{\mathfrak{f}^0} ; e\notin A}(\cdot)
=\mathfrak{T}_{ \mathcal{G}_{\mathfrak{f}^0} \vee e;(\alpha_3,\alpha_4)} $ should be direct. Focus now on the second sum.

Consider a trivial $p$-inner loop $e$ in $ \mathcal{G}_{\mathfrak{f}^0}$. Then $e$ is a trivial $p$-inner loop in all $A\Subset \mathcal{G}_{\mathfrak{f}^0}$ containing~$e$. Contracting~$e$ generates~$p$ discs and $4-p$ non trivial vertices.
Using Lemma~\ref{lem:tris}, the nullity is $n(A) = n(A') + 1$ which ensures the factor~$y-1$, and the exponent of~$z$ becomes
\begin{gather*}
9k(A) - \big[6(V( \mathcal{G}_{\mathfrak{f}^0}) - E(A)) + 3F_{\rm int}(A) - 2(1+\alpha_3)B^3(A)+(3\alpha_3-\alpha_4)B^4(A)\big]\\
\qquad{} = 9(k(A')-3)
- \big[6[V(\mathcal{G}_{\mathfrak{f}^0}/e)-3) - (E(A')+1)]
 + 3F_{\rm int}(A')\\
\qquad\quad{}- 2(1+\alpha_3)(B^3(A')-\beta_p) +
(3\alpha_3-\alpha_4)(B^4(A')-\beta'_p)\big] \\
\qquad{} = 9k(A') - \big[6(V(\mathcal{G}_{\mathfrak{f}^0}) - E(A')) + 3F_{\rm int}(A') - 2(1+\alpha_3)B^3(A')\\
\qquad\quad{}
+
(3\alpha_3-\alpha_4)B^4(A')\big] -15+6p+3\alpha_3(4-p)+\alpha_4(3p-8),
\end{gather*}
where $\beta_p=6-3p$, $\beta'_p=8-3p$, $A$ and $A'$ refer to the subgraphs
related by bijection between spanning c-subgraphs of $\{A \Subset \mathcal{G}_{\mathfrak{f}^0};e\in A\}$ and $\{A'\Subset \mathcal{G}_{\mathfrak{f}^0}/e\}$.
Finally, one gets $(y-1)z^{-15+6p+3\alpha_3(4-p)+\alpha_4(3p-8)}\mathfrak{T}_{ \mathcal{G}_{\mathfrak{f}^0}/e;(\alpha_3,\alpha_4)}$. As previously shows, this result is independent of the representative.
\end{proof}

Let us come back on the expression of our polynomial to make a comparison with recent results introduced in~\cite{Tanasa:2010me}. One variable (the variable~$s$) stands for the number of connected components of the boundary graph of a~spanning c-subgraph. We replace this number by~$\sum_{{\mathbf{b}}^3} C_{\partial}({\mathbf{b}}^3)$. With the change of variables $s=z^{-1}$, we directly recover the sum of genera as introduced by Tanasa~\cite{Tanasa:2010me} on some special closed tensor graph (with, necessarily, all 3-bubbles closed). However, there is a relationship between the number of boundary components of a~graph~$\mathcal{G}$ denoted $C_{\partial}(\mathcal{G})$ and the sum $\sum_{{\mathbf{b}}^3} C_{\partial}({\mathbf{b}}^3)$. This certainly needs to be investigated in the context of w-colored stranded graphs.

\section[The polynomial invariant for $nD$ w-colored graphs]{The polynomial invariant for $\boldsymbol{nD}$ w-colored graphs}
\label{subset:polymtn}

Having identif\/ied a rank 4 invariant polynomial, we can quickly generalize it in any rank starting by extending the previous results.

\begin{Lemma}\label{lemma:bulbesn}
Consider a representative $\mathcal{G}_{\mathfrak{f}^0}$ of a rank $n$ w-colored graph $\mathfrak{G}$. $\forall$ $3\leq p\leq d\leq n$,
\begin{gather} \label{eq:vefsumn}
\sum_{{\mathbf{b}}^d \in \mathcal{B}^d} V_{{\mathbf{b}}^d} \geq \complement_{n}^{d-1}V, \qquad \sum_{{\mathbf{b}}^d \in \mathcal{B}^d} E_{{\mathbf{b}}^d} = \complement_{n}^{d-1}E, \qquad \sum_{{\mathbf{b}}^d\in \mathcal{B}^d} F_{{\rm int};{\mathbf{b}}^d} = \complement_{n-1}^{d-2}F_{{\rm int}},
\\
 \label{eq:sumdoublen}
\sum_{{\mathbf{b}}^d \in \mathcal{B}^d} B^p_{{\mathbf{b}}^d}=\complement_{n-p+1}^{d-p}B^p, \qquad \complement_{d-1}^{p-1}B^d\leq \complement_{n-p+1}^{d-p}B^p.
\end{gather}
\end{Lemma}
\begin{proof}
Consider a vertex $v$ of a representative $\mathcal{G}_{\mathfrak{f}^0}$ of a rank $n$ w-colored graph; $v$ can be decomposed, at least, in $\complement_{n}^{d-1}$ ($d\leq n$) vertices which could belong to a $d$-bubble (this number is the minimum given by the simplest vertex of the graph $ \mathcal{G}_{1,\mathfrak{f}^0(\mathcal{G}_1)}$ in Fig.~\ref{fig:npbn}). We then prove the f\/irst relation in~\eqref{eq:vefsumn}. From the coloring, each edge of $ \mathcal{G}_{\mathfrak{f}^0}$ decomposes into $\complement_{n}^{d-1}$ edges belonging to a~$d$-bubble. Each internal face of $ \mathcal{G}_{\mathfrak{f}^0}$ is shared by $\complement_{n-1}^{d-2}$ $d$-bubbles. Thus we have the remaining equations in~\eqref{eq:vefsumn}.

Let us prove \eqref{eq:sumdoublen}. Consider the $d$-bubbles in $\mathcal{G}_{\mathfrak{f}^0}$. Each $p$-bubble ${\mathbf{b}}^p$ ($p\leq d$) in $\mathcal{G}_{\mathfrak{f}^0}$, is a $p$-bubble of $\complement_{n-p+1}^{d-p}$ $d$-bubbles of $ \mathcal{G}_{\mathfrak{f}^0}$.
Indeed, consider a $p$-bubble in $ \mathcal{G}_{\mathfrak{f}^0}$, it remains $n-p+1$ colors from which we can take arbitrary $d-p$ colors to obtain a $d$-bubble. The f\/irst equation in~\eqref{eq:sumdoublen} follows.
Using once again the coloring, each $d$-bubble has at least $ \complement_{d-1}^{p-1}$ $p$-bubbles, that is $B^p_{{\mathbf{b}}^d}\geq \complement_{d-1}^{p-1}$. We can now sum each member of this relation on the total number of $d$-bubbles and complete the proof of~\eqref{eq:sumdoublen}.
\end{proof}

\begin{Proposition}
Let $ \mathcal{G}_{\mathfrak{f}^0}$ be any representative of a~rank $n\geq 4$ w-colored graph. There exist some positive rational numbers $\boldsymbol{\alpha}=\{\alpha_k\}_{k=3, \dots, n}$ such that
\begin{gather*} 
\gamma_{n;\boldsymbol{\alpha}}( \mathcal{G}_{\mathfrak{f}^0})=\frac{n(n-1)}{2}(V-E) + (n-1)F_{{\rm int}} - (2+(n-2)\alpha_3)B^3 \\
\hphantom{\gamma_{n;\boldsymbol{\alpha}}( \mathcal{G}_{\mathfrak{f}^0})=}{} + \sum_{k=4}^n \big[(k-1)\alpha_{k-1}-(n-k+1)\alpha_k\big]B^k
\end{gather*}
is a negative integer.
\end{Proposition}

\begin{proof} Consider a representative of a rank $n$ w-colored graph. The $3$-bubbles of $\mathcal{G}_{\mathfrak{f}^0}$ are ribbon graphs. Each of them satisf\/ies \eqref{eq:forran2} and the sum over all the $3$-bubbles of $ \mathcal{G}_{\mathfrak{f}^0}$ gives
\begin{gather*}
\sum_{{\mathbf{b}}^3 \in \mathcal{B}^3}(V_{{\mathbf{b}}^3}-E_{{\mathbf{b}}^3})+ \sum_{{\mathbf{b}}^3 \in \mathcal{B}^3}F_{{\rm int};{\mathbf{b}}^3} -2B^3\leq 0.
\end{gather*}
From Lemma \ref{lemma:bulbesn}, we have
\begin{gather}
\frac{n(n-1)}{2}(V-E) + (n-1)F_{{\rm int}} - 2B^3\leq 0,\nonumber\\
kB^{k+1} - (n-k+1)B^k \leq 0, \qquad \forall\, k=3,\dots, n-1, \qquad -B^n+n\leq0.\label{eq:rrannpp}
\end{gather}
We multiply the second relation in \eqref{eq:rrannpp} by a positive integer $\alpha_k$, $k=3,\dots,n-1$ and the last relation by a positive integer $\alpha_n$ such that the quantity
\begin{gather*}
- (2+(n-2)\alpha_3)B^3 + \sum_{k=4}^n \big[(k-1)\alpha_{k-1}-(n-k+1)\alpha_k\big]B^k
\end{gather*}
 is an integer. We obtain
\begin{gather} \label{eq:summa}
 k\alpha_kB^{k+1} - (n-k+1)\alpha_kB^k \leq 0, \qquad \forall\, k=3,\dots, n-1, \qquad -\alpha_nB^n+n\alpha_n\leq0.
\end{gather}
A summation from $k=3$ to $k=n-1$ of the f\/irst relation in~\eqref{eq:summa} and then adding the result to the second inequality leads to
\begin{gather}
\label{eq:summa2}
\sum_{k=3}^{n-1}\big[k\alpha_kB^{k+1} - (n-k+1)\alpha_kB^k\big]-\alpha_{n}B^{n}+n\alpha_{n}\leq 0.
\end{gather}
Adding relation \eqref{eq:summa2} to the f\/irst relation in~\eqref{eq:rrannpp} achieves the proof.
\end{proof}

One may wonder why the number $B^2$ is not used in the above summation. $B^2$ denotes the number of 2-bubbles which is equal to $F_{{\rm int}}+C_\partial$. One notes that the quantity~$C_\partial$ is not used in the invariant, we have introduced a dif\/ferent variable for internal faces and boundary components.

The number $\gamma_{n;\boldsymbol{\alpha}}(\cdot)$ is a generalized and weighted Euler characteristics. If one envisages to have an alternating sum of positive integers, we must require the following conditions:
\begin{gather*}
 (k-1)\alpha_{k-1}-(n-k+1)\alpha_k>0 \qquad \mbox{if $k$ is even},\\
(k-1)\alpha_{k-1}-(n-k+1)\alpha_k<0 \qquad \mbox{if $k$ is odd}.
\end{gather*}
This system which is triangular can be clearly solved and generally admits several solutions.

It is instructive to investigate how the above invariant $\gamma_{n;\boldsymbol{\alpha}}(\cdot)$ can be prolonged to the lowest ranks, $n=2$ and $n=3$. Consider a representative $ \mathcal{G}_{\mathfrak{f}^0}$ of rank $n$ w-colored stranded graphs~$\mathfrak{G}$.

For $n=2$, $ \mathcal{G}_{\mathfrak{f}^0}$ is a half-edged ribbon graph. Consider the underlying ribbon graph $\mathcal{G}$, or ribbon graph in the sense of Bollob\'as and Riordan, the invariant used is $2k(A)-(V(A)-E(A)+F(A))$ for each $A\Subset \mathcal{G}$. In the case of half-edged ribbon graphs, the invariant is $2k(A)-(V(A)-E(A)+F_{{\rm int}}(A))$ for each $A\Subset \mathcal{G}_{\mathfrak{f}^0}$. Since this quantity is always non negative, we have for a connected graph~$ \mathcal{G}_{\mathfrak{f}^0}$,
\begin{gather} \label{eq:rran2}
V( \mathcal{G}_{\mathfrak{f}^0})-E( \mathcal{G}_{\mathfrak{f}^0})+ F_{{\rm int}}(\mathcal{G}_{\mathfrak{f}^0})\leq 2,
\end{gather}
where $V$, $E$ and $F_{{\rm int}}$ are, respectively, the number of vertices, edges and internal faces. We set
\begin{gather*}
\gamma_2( \mathcal{G}_{\mathfrak{f}^0})=V-E+F_{{\rm int}}.
\end{gather*}
This is precisely the invariant used in the def\/inition of the BR polynomial on ribbon graphs with half-edges~\cite{remia}.

For $n=3$, the $3$-bubbles of $ \mathcal{G}_{\mathfrak{f}^0}$ are connected rank~$2$ or ribbon graphs. Each of them satisf\/ies~\eqref{eq:rran2} and a~summation on all the $3$-bubbles of $ \mathcal{G}_{\mathfrak{f}^0}$ gives
\begin{gather*}
\sum_{{\mathbf{b}}^3 \in \mathcal{B}^3}(V_{{\mathbf{b}}^3}-E_{{\mathbf{b}}^3})+ \sum_{{\mathbf{b}}^3 \in \mathcal{B}^3}F_{{\rm int};{\mathbf{b}}^3} -2B^3\leq 0.
\end{gather*}
Using now Lemma~\ref{lemma:bulbesn}, we have
\begin{gather} \label{eq:rrank3}
3(V-E) + 2F_{{\rm int}} - 2B^3\leq 0, \qquad -B^3+3\leq 0.
\end{gather}
Consider a positive rational number $\alpha_3$ such that $(2+\alpha_3)B^3$ is an integer. We multiply the second relation in~\eqref{eq:rrank3} by $\alpha_3$ and obtain $-\alpha_3B^3+3\alpha_3\leq0$. We add this relation to the f\/irst relation in~\eqref{eq:rrank3} and get
\begin{gather} \label{eq:rran3}
3(V-E) + 2F_{{\rm int}} - (2+\alpha_3)B^3\leq-3\alpha_3.
\end{gather}
We can then set
\begin{gather*}
\gamma_{3;\alpha_3}(\mathcal{G})=3(V-E) + 2F_{{\rm int}}- (2+\alpha_3)B^3.
\end{gather*}
The expression $\gamma_{3;0}(\cdot)$ is the invariant introduced on rank~3 w-colored stranded graphs in~\cite{remia}. Hence, the present framework fully extends that work to a~one-parameter family of invariant polynomials.

\begin{Definition}[topological invariant for rank $n$ w-colored
graph]
Let $\mathfrak{G}(\mathcal{V},\mathcal{E},\mathfrak{f}^0)$ be a rank~$n$ w-colored graph and $\boldsymbol{\alpha}=\{\alpha_k\}_{k=3,\dots, n}$ some positive rational numbers.
The generalized topological invariant associated with~$\mathfrak{G}$
is given by the following function associated with any of its
representatives~$\mathcal{G}_{\mathfrak{f}^0}$.
\begin{gather}
\mathfrak{T}_{\mathfrak{G};\boldsymbol{\alpha}}(x,y,z,s,w,q,t) =\mathfrak{T}_{\mathcal{G}_{\mathfrak{f}^0};\boldsymbol{\alpha}}(x,y,z,s,w,q,t)\nonumber\\
= \sum_{A \Subset \mathcal{G}_{\mathfrak{f}^0}}
 (x-1)^{\rk( \mathcal{G}_{\mathfrak{f}^0})-\rk(A)}(y-1)^{n(A)}
z^{\frac{(n-1)(n+2)}{2}k(A)-\gamma_{n;\boldsymbol{\alpha}}(A)}s^{C_\partial(A)} w^{F_{\partial}(A)} q^{E_{\partial}(A)} t^{f(A)},
\label{ttopflan}
\end{gather}
with
\begin{gather*}
\gamma_{n;\boldsymbol{\alpha}}(A)=\frac{n(n-1)}{2}(V(A)-E(A)) + (n-1)F_{{\rm int}}(A) - (2+(n-2)\alpha_3)B^3(A)\\
\hphantom{\gamma_{n;\boldsymbol{\alpha}}(A)=}{} + \sum_{k=4}^n \big[(k-1)\alpha_{k-1}-(n-k+1)\alpha_k\big]B^k(A)
\end{gather*}
a negative integer.
\end{Definition}

Remark that Lemma~\ref{lemma:bulbes4} f\/inds an extension for any rank $n$ w-colored graph.

\begin{figure}[h]
 \centering

\includegraphics[angle=0, width=6cm, height=1cm]{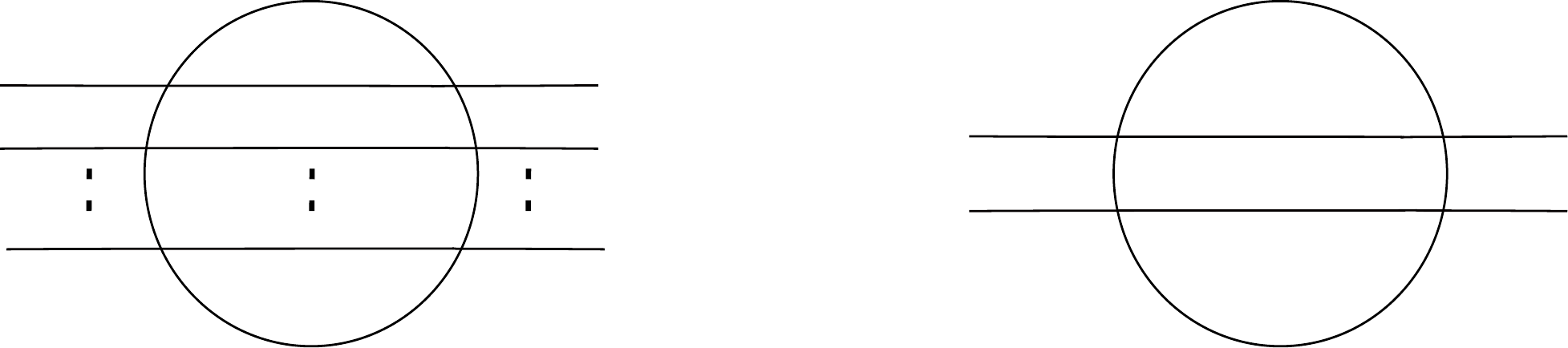}
\put(-152,-10){$ \mathcal{G}_{1,\mathfrak{f}^0(\mathcal{G}_1)}$}
\put(-101,10){$0$}
\put(-178,10){$\bar 0$}
\put(-73,10){$\bar 0$}
\put(-33,-10){${\mathbf{b}}$}
\put(5,10){$0$}

\caption{A 3-bubble ${\mathbf{b}}$ of the graph $\mathcal{G}_{1,\mathfrak{f}^0(\mathcal{G}_1)}$ made with $n$ strands.} \label{fig:npbn}
\end{figure}
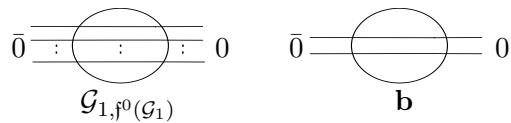

Let us introduce the following lemma used in the proof of our main result.

\begin{Lemma}[cut/contraction of special edges]\label{lem:cutbrinn}
Let $ \mathcal{G}_{\mathfrak{f}^0}$ be a representative of~$\mathfrak{G}$ a rank n w-colored graph
and $e$ an edge in $\mathcal{G}_{\mathfrak{f}^0}$. Then, if $e$ is a bridge in the above notations, we have
\begin{gather}
k(\mathcal{G}_{\mathfrak{f}^0}\vee e) = k( \mathcal{G}_{\mathfrak{f}^0}/e)+1,\qquad
V(\mathcal{G}_{\mathfrak{f}^0}\vee e) = V(\mathcal{G}_{\mathfrak{f}^0}/e) + 1,\qquad
E( \mathcal{G}_{\mathfrak{f}^0}\vee e) = E( \mathcal{G}_{\mathfrak{f}^0}/e),\nonumber\\
f(\mathcal{G}_{\mathfrak{f}^0} \vee e)= f( \mathcal{G}_{\mathfrak{f}^0}/e) +2, \label{eq:kven} \\
F_{{\rm int}}(\mathcal{G}_{\mathfrak{f}^0}\vee e)= F_{{\rm int}}( \mathcal{G}_{\mathfrak{f}^0}/e), \qquad
B^p_{{\rm int}}(\mathcal{G}_{\mathfrak{f}^0} \vee e)=B^p_{{\rm int}}( \mathcal{G}_{\mathfrak{f}^0}/e), \qquad
\forall\, 3\leq p\leq n,\label{eq:fintn} \\
C_\partial( \mathcal{G}_{\mathfrak{f}^0} \vee e) = C_\partial( \mathcal{G}_{\mathfrak{f}^0}/e) +1,
\qquad
E_\partial( \mathcal{G}_{\mathfrak{f}^0} \vee e) = E_\partial ( \mathcal{G}_{\mathfrak{f}^0}/e) +n,\nonumber\\
F_\partial( \mathcal{G}_{\mathfrak{f}^0} \vee e) = F_\partial(\mathcal{G}_{\mathfrak{f}^0}/e) + \complement_n^2,\label{eq:cextbordn} \\
 B^p_{{\rm ext}}(\mathcal{G}_{\mathfrak{f}^0}\vee e) = B^p_{{\rm ext}}(\mathcal{G}_{\mathfrak{f}^0}/e) + \complement_n^{n-p+1},\qquad
 \forall\, 3\leq p\leq n.
\label{eq:bubextn}
\end{gather}
\end{Lemma}

\begin{proof} This is a simple extension of Lemma~\ref{lem:cutbri}. We f\/irst concentrate on the case of the bridge.
In~\eqref{eq:kven}, the equations can be easily found.
We now investigate~\eqref{eq:fintn}. Notice that
the faces passing through $e$ are necessarily open by~\cite[Lemma~7]{remia}. All closed faces on each side of the bridge are preserved after cutting $e$ and are still preserved after
edge contraction. Hence $F_{{\rm int}}(\mathcal{G}_{\mathfrak{f}^0}\vee e) = F_{{\rm int}}( \mathcal{G}_{\mathfrak{f}^0}/e)$,
$B^p_{{\rm int}}( \mathcal{G}_{\mathfrak{f}^0} \vee e)=B^p_{{\rm int}}( \mathcal{G}_{\mathfrak{f}^0}/e)$ $\forall$ $3\leq p\leq n $.
Let us focus on~\eqref{eq:cextbordn} and use the fact that the $n$ external faces belong to the same boundary component. By cutting $e$, this unique component yields two boundary components. It is direct to get
$C_\partial( \mathcal{G}_{\mathfrak{f}^0} \vee e) = C_\partial( \mathcal{G}_{\mathfrak{f}^0}/e) +1$,
$E_\partial(\mathcal{G}_{\mathfrak{f}^0} \vee e)= E_\partial (\mathcal{G}_{\mathfrak{f}^0}/e)+n$
(the cut of $e$ divides each external face into two dif\/ferent
external strands) and $F_\partial( \mathcal{G}_{\mathfrak{f}^0} \vee e)= F_\partial ( \mathcal{G}_{\mathfrak{f}^0}/e)+\complement_n^2$
since $C_\partial( \mathcal{G}_{\mathfrak{f}^0}/e)=C_\partial( \mathcal{G}_{\mathfrak{f}^0})$,
$E_\partial ( \mathcal{G}_{\mathfrak{f}^0}/e)=E_\partial ( \mathcal{G}_{\mathfrak{f}^0})$
and $F_\partial( \mathcal{G}_{\mathfrak{f}^0}G /e)= F_\partial ( \mathcal{G}_{\mathfrak{f}^0})$
 from Lemma~\ref{lem:fullcont}.
For the number of external bubbles,
there are $\complement_n^{n-p+1}$ $p$-bubbles in $ \mathcal{G}_{\mathfrak{f}^0}$ passing through the
bridge. These bubbles are clearly in~$ \mathcal{G}_{\mathfrak{f}^0}/e$ and, cutting the bridge,
each of these bubbles splits in two. This yields~\eqref{eq:bubextn}.
\end{proof}

We can now introduce our main result of this section.

\begin{Theorem}[contraction/cut rule for w-colored graphs]
\label{theo:contensn}
Let $\mathfrak{G}$ be a rank $n$ w-colored graph and and $\boldsymbol{\alpha}=\{\alpha_k\}_{k=3,\dots, n}$ some positive rational numbers. Then, for a regular edge $e$ of any of the representative $ \mathcal{G}_{\mathfrak{f}^0}$ of $\mathfrak{G}$, we have
\begin{gather}
\mathfrak{T}_{ \mathcal{G}_{\mathfrak{f}^0};\boldsymbol{\alpha}}=\mathfrak{T}_{ \mathcal{G}_{\mathfrak{f}^0}\vee e;\boldsymbol{\alpha}} +\mathfrak{T}_{ \mathcal{G}_{\mathfrak{f}^0}/e;\boldsymbol{\alpha}}.
\label{tenscondeln}
\end{gather}
For a bridge $e$, we have
\begin{gather*}
\mathfrak{T}_{\mathcal{G}_{\mathfrak{f}^0} \vee e;\boldsymbol{\alpha}}= z^{(n-1)(n+1)+n\alpha_n}sw^{\frac{n(n-1)}{2}}q^nt^2 \mathfrak{T}_{\mathcal{G}_{\mathfrak{f}^0}/e;\boldsymbol{\alpha}}
\end{gather*}
and
\begin{gather*}
\mathfrak{T}_{ \mathcal{G}_{\mathfrak{f}^0};\boldsymbol{\alpha}} =\big[(x-1)z^{(n-1)(n+1)+n\alpha_n}sw^{\frac{n(n-1)}{2}}q^nt^2+1\big] \mathfrak{T}_{ \mathcal{G}_{\mathfrak{f}^0}/e;\boldsymbol{\alpha}}.
\end{gather*}
\end{Theorem}

This proof is an extended form of the proof given in Theorem~\ref{theo:contens4} but with a more larger number of bubbles. We see that in addition to the (3,4)-bubbles, we have also some $i$-bubbles ($i\geq 4$).

\begin{proof}
Consider a representative $ \mathcal{G}_{\mathfrak{f}^0}$ of a rank $n$ w-colored graph $\mathfrak{G}$. Following the proof of~\eqref{tenscondel} in Theorem~\ref{theo:contens4} and replacing $(3,4)$-bubbles by $i$-bubbles ($i\geq 3$), \eqref{tenscondeln} is direct. Moreover, the same arguments for proving the bridge relation and~\eqref{tensbri}, using Lemma~\ref{lem:cutbrinn} instead of Lemma~\ref{lem:cutbri} allow us to achieve the proof.
\end{proof}

Theorem \ref{theo:contensn} shows that the rank $D$ w-colored graph polynomial satisf\/ies the recurrence relation of contraction and cut of ordinary edges. For $n=3$, $\alpha_3=0$, we get Theorem \ref{theo:contens3} and $n=4$ leads to Theorem \ref{theo:contens4}. Concerning special edges, we have only discussed the bridge case. There are certainly some relations for trivial $p$-inner self-loops ($p=0,\dots, n$) but these relations are numerous and lengthy and do not add much to the discussion.

The reductions of the above polynomial \eqref{ttopflan} to $\mathfrak{T}'_{ \mathcal{G}_{\mathfrak{f}^0};\boldsymbol{\alpha}}$, $\mathfrak{T}''_{ \mathcal{G}_{\mathfrak{f}^0};\boldsymbol{\alpha}}$ and $\mathfrak{T}'''_{\mathcal{G}_{\mathfrak{f}^0};\boldsymbol{\alpha}}$ also satisfy the contraction/cut rule. Furthermore $\mathfrak{T}_{\mathcal{G}_{\mathfrak{f}^0};\boldsymbol{\alpha}}$ maps to Tutte polynomial by taking $n=1$ and by putting the variables~$z$,~$w$, $q$ and~$t$ to~1. Taking $n=2$ and using the mapping between the half-edged ribbon graphs and the rank~2 w-colored graph, we directly f\/ind the polynomial invariant on ribbon graphs with half-edges~\cite{remia}. An appropriate change of variable $s=z^{-1}$ gives the BR polynomial.

Consider a representative $\mathcal{G}_{\mathfrak{f}^0}$ of rank $n$ w-colored graph $\mathfrak{G}$.
In the polynomial invariant introduced in \eqref{ttopflan}, we can add other variables $y_i$ ($i=0,\dots, n$) for the number of $i$-bubbles of the boundary graph $\partial \mathcal{G}_{\mathfrak{f}^0}$. An extended form of the polynomial introduced in \eqref{ttopflan} is given by
\begin{gather}
\mathfrak{T}_{\mathfrak{G};\boldsymbol{\alpha}}(x,y,z,\{y_i\}_{i=0,\dots, n}) =\mathfrak{T}_{\mathcal{G}_{\mathfrak{f}^0};\boldsymbol{\alpha}}(x,y,z,\{y_i\}_{i=0,\dots, n})\nonumber \\
 \qquad{} =\sum_{A \Subset \mathcal{G}_{\mathfrak{f}^0}}
 (x-1)^{\rk( \mathcal{G}_{\mathfrak{f}^0})-\rk(A)}(y-1)^{n(A)}
z^{\frac{(n-1)(n+2)}{2}k(A)-\gamma_{n;\boldsymbol{\alpha}}(A)} \left(\prod_{i=0}^{n}y_i^{B^i(\partial(A))}\right).\label{ttopflanext}
\end{gather}

\begin{Definition}[multivariate form]
The multivariate form associated with \eqref{ttopflanext} is def\/ined by:
\begin{gather}
\widetilde{\mathfrak{T}}_{\mathfrak{G}}\big(x,\{\beta_e\},\{z_i\}_{i=2,\dots, n-1},\{y_i\}_{i=0,\dots, n},\{q_i\}_{i=3,\dots, n}\big)\nonumber\\
\qquad {}=
\widetilde{\mathfrak{T}}_{ \mathcal{G}_{\mathfrak{f}^0}}\big(x,\{\beta_e\},\{z_i\}_{i=2,\dots, n-1},\{y_i\}_{i=0,\dots, n},\{q_i\}_{i=3,\dots, n}\big)\nonumber \\ 
\qquad {}=
 \sum_{A \Subset \mathcal{G}_{\mathfrak{f}^0}}
 x^{\rk(A)}\left(\prod_{e\in A}\beta_e\right)
\left(\prod_{i=2}^{n}z_i^{B_{{\rm int}}^i(A)}\right)\left(\prod_{i=3}^{n}z_i^{B_{{\rm ext}}^i(A)}\right)
\left(\prod_{i=0}^{n}y_i^{B^i(\partial(A))}\right),\label{multi}
\end{gather}
for $\{\beta_e\}_{e\in \mathcal{E}}$ labeling the edges of the graph $ \mathcal{G}_{\mathfrak{f}^0}$.
\end{Definition}

This multivariate form extends the polynomial of Gurau introduced in~\cite{Gurau:2009tz} in an essential way. Indeed, in addition to the internal bubbles, we also deal with external ones and the multivariate form~\eqref{multi} is def\/ined on an extended category of graphs.

The following holds for all non loop edge $e$
\begin{gather*}
 \widetilde{\mathfrak{T}}_{\mathfrak{G}}=
\widetilde{\mathfrak{T}}_{ \mathfrak{G}\vee e} +x\beta_e \widetilde{\mathfrak{T}}_{\mathfrak{G}/e}.
\end{gather*}

\section{Example}
 Consider the graph $\mathcal{G}$ of Fig.~\ref{fig:cg2}. Let us prove that the polynomial invariant $\mathfrak{T}_{\mathcal{G};(\alpha_3,\alpha_4)}$ satisf\/ies the recurrence relation.

\begin{figure}[h] \centering
\includegraphics[angle=0, width=8cm, height=3cm]{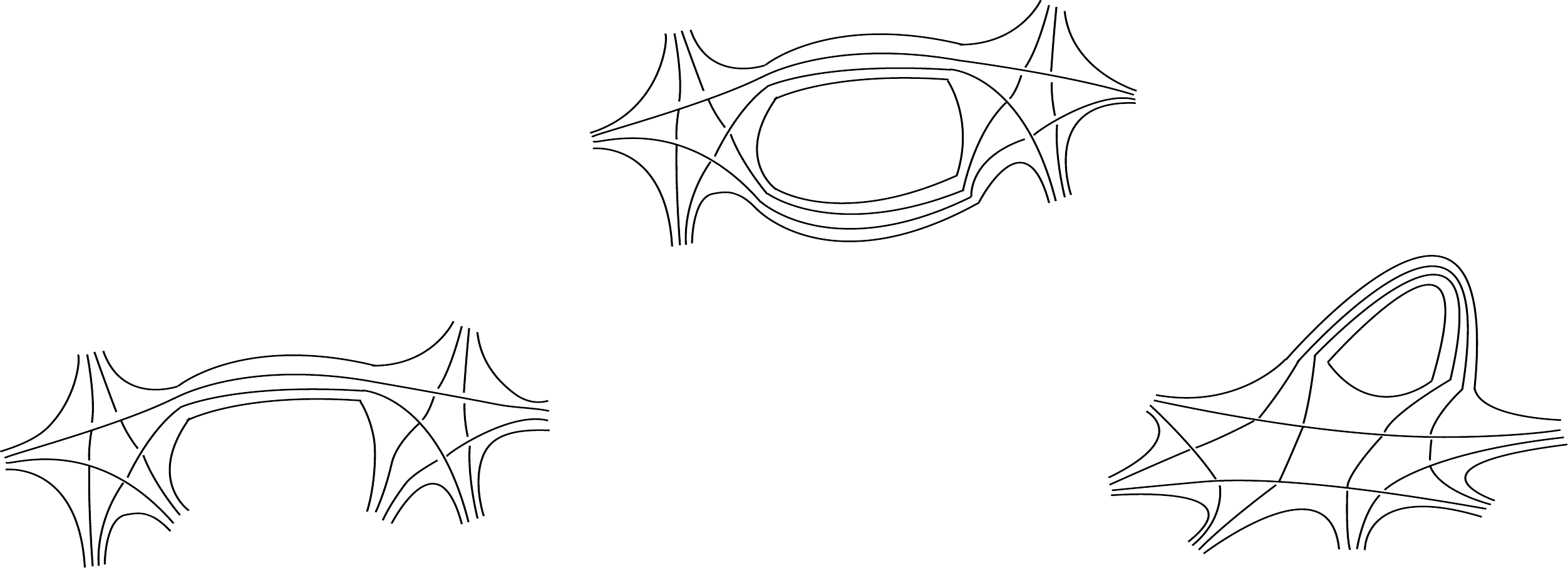}
\put(-107,82){\small{$e_1$}}
\put(-118,80){\small{$1$}}
\put(-96,80){\small{$\bar1$}}
\put(-138,77){\small{$0$}}
\put(-72,80){\small{$\bar0$}}
\put(-116,58){\small{$2$}}
\put(-96,59){\small{$\bar2$}}
\put(-127,46){\small{$3$}}
\put(-72,50){\small{$\bar3$}}
\put(-149,65){\small{$4$}}
\put(-62,68){\small{$\bar4$}}
\put(-204,32){\small{$1$}}
\put(-179,32){\small{$\bar1$}}
\put(-224,32){\small{$0$}}
\put(-158,35){\small{$\bar0$}}
\put(-201,10){\small{$2$}}
\put(-181,10){\small{$\bar2$}}
\put(-212,-3){\small{$3$}}
\put(-157,0){\small{$\bar3$}}
\put(-234,15){\small{$4$}}
\put(-147,20){\small{$\bar4$}}
\put(-43,37){\small{$1$}}
\put(-13,35){\small{$\bar1$}}
\put(-53,28){\small{$0$}}
\put(-3,25){\small{$\bar0$}}
\put(-73,10){\small{$4$}}
\put(-10,6){\small{$\bar4$}}
\put(-62,0){\small{$3$}}
\put(-28,-2){\small{$\bar3$}}
\put(-29,50){\small{$e_1$}}
\put(-190,35){\small{$e_1$}}
\put(-105,43){\small{$e_2$}}
\put(-205,-17){$\mathcal{G}_{\mathfrak{f}^0} \vee e_2$}
\put(-50,-17){$\mathcal{G}_{\mathfrak{f}^0}/e_2$}
\put(-105,30){$\mathcal{G}_{\mathfrak{f}^0}$}

\caption{A representative $ \mathcal{G}_{\mathfrak{f}^0}$ of a w-colored graph, the cut graph $ \mathcal{G}_{\mathfrak{f}^0}\vee e_2$ and the contracted graph $\mathcal{G}_{\mathfrak{f}^0}/e_2$
with respect to $e_2$.}\label{fig:cg2}
\end{figure}
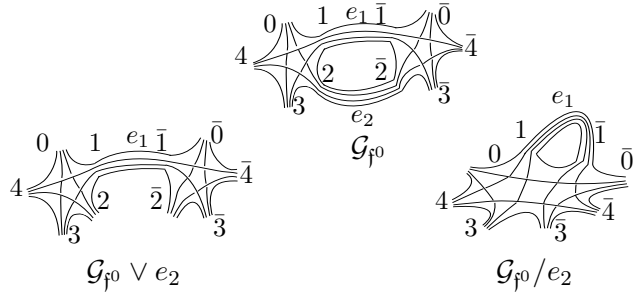

Using the spanning c-subgraph summation, we get
\begin{gather*}
\mathfrak{T}_{\mathcal{G}_{\mathfrak{f}^0};(\alpha_3,\alpha_4)}(x,y,z,s,w,q,t) = (y-1)z^{28+7\alpha_3+5\alpha_4} s w^{11}q^{12} t^{6}
+ 2z^{31+10\alpha_3+6\alpha_4} s w^{14}q^{16} t^{8}\\
\hphantom{\mathfrak{T}_{\mathcal{G}_{\mathfrak{f}^0};(\alpha_3,\alpha_4)}(x,y,z,s,w,q,t) =}{}
 + (x-1)z^{37+10\alpha_3+10\alpha_4}s^2w^{20}q^{20} t^{10}.
\end{gather*}
We must check that
\begin{gather*}
\mathfrak{T}_{ \mathcal{G}_{\mathfrak{f}^0};(\alpha_3,\alpha_4)} = \mathfrak{T}_{ \mathcal{G}_{\mathfrak{f}^0}\vee e_2;(\alpha_3,\alpha_4)} +\mathfrak{T}_{ \mathcal{G}_{\mathfrak{f}^0}/e_2;(\alpha_3,\alpha_4)}.
\end{gather*}
This is the case, since
\begin{gather*}
\mathfrak{T}_{ \mathcal{G}_{\mathfrak{f}^0}\vee e_2;(\alpha_3,\alpha_4)}(x,y,z,s,w,q,t) =z^{31+10\alpha_3+6\alpha_4} s w^{14}q^{16} t^{8}\\
\hphantom{\mathfrak{T}_{ \mathcal{G}_{\mathfrak{f}^0}\vee e_2;(\alpha_3,\alpha_4)}(x,y,z,s,w,q,t) =}{}
 + (x-1)z^{37+10\alpha_3+10\alpha_4}s^2w^{20}q^{20} t^{10},
\end{gather*}
and
\begin{gather*}
\mathfrak{T}_{ \mathcal{G}_{\mathfrak{f}^0}/e_2;(\alpha_3,\alpha_4)}(x,y,z,s,w,q,t) =(y-1)z^{28+7\alpha_3+5\alpha_4} s w^{11}q^{12} t^{6}+z^{31+10\alpha_3+6\alpha_4} s w^{14}q^{16} t^{8}.
\end{gather*}

The dif\/ferent expressions of $\gamma_{4;(\alpha_3,\alpha_4)}(\cdot)$ in the above polynomial are $-19-7\alpha_3-5\alpha_4$, $-22-10\alpha_3-6\alpha_4$ and $-28-10\alpha_3-10\alpha_4$. Putting $\alpha_3=\frac{1}{2}$ and $\alpha_4=\frac{3}{2}$, each of these expressions is a negative integer.

As a f\/inal remark, the parameters $\alpha_i$ which label the $\gamma_n$ invariant are positive rational numbers. We may ask a unique prescription to determine those parameters to get a unique
$\gamma_n$. In the non integer rational and integer cases,
we do have the following issues to solve this problem:

In the non integer rational numbers case, we have tuned $\alpha_3 = 1/2$ and $\alpha_4= 3/2$ in the above example before getting $\gamma_n$ as an integer for all subgraphs. Thus, for each graph, we claim that we have $2^E$ constraints to solve for f\/inding $n-2$ rational coef\/f\/icients. The procedure might not in general admit a solution (usually the case when the system is over-determined, here this happens whenever $n < 2(2^{E-1} +1)$ or when the rank do not evolve fast enough with respect to the number of edges of the graph). Furthermore, this solution is graph dependent and we do not see how to generalize it to the arbitrary case. We quickly realize that, in contrast, restricted to integer coef\/f\/icients, the problem has several solutions and it is graph independent. In that case, the issue becomes dif\/ferent: we do have too many ways to def\/ine the invariant. We must therefore think about a new way to f\/ix this object.

Consider the parameters $\alpha_i$ as integers. Among possible invariants, one may consider the one with minimum value simply because it will be the apparently less complicated to write down. Interestingly, we face here a known problem in optimization analysis. This is an integer programming problem~\cite{Wolsey} of the form
\begin{gather*}
\min \big\{ Bx\leq 0, \, x\in \mathbb{Z}^{n+1}\big\},
\end{gather*}
where $B=(B^0,\dots, B^n)\in \mathbb{N}^{n+1}$ and $x=(a_0\alpha_0,\dots, a_n\alpha_n)$ with $a_i$ the binomial coef\/f\/icient appearing in $\gamma_n$. Thus studying this kind of problem could be interesting to f\/ix our parameters and uniquely determine the $\gamma_n$ invariant. This deserves to be elucidated.

\appendix

\section[Determination of the invariant by the sum of invariants of consecutive bubbles]{Determination of the invariant by the sum of invariants\\
of consecutive bubbles}
\label{appb}

Consider a representative $ \mathcal{G}_{\mathfrak{f}^0}$ of a rank $n$ w-colored stranded graphs $\mathfrak{G}$.

For $n=2,3$, this method is not dif\/ferent from the previous one.

Case $n=4$. The $4$-bubbles of $ \mathcal{G}_{\mathfrak{f}^0}$ are rank $3$ graphs. Each of them satisf\/ies \eqref{eq:rran3} and a sum over all these $4$-bubbles gives
\begin{gather*}
3\sum_{{\mathbf{b}}^4 \in \mathcal{B}^4}(V_{{\mathbf{b}}^4}-E_{{\mathbf{b}}^4})+2 \sum_{{\mathbf{b}}^4 \in \mathcal{B}^4}F_{{\rm int};{\mathbf{b}}^4} - (2+\alpha_3)\sum_{{\mathbf{b}}^4 \in \mathcal{B}^4}B^3_{{\mathbf{b}}^4}+3\alpha_3B^4\leq 0.
\end{gather*}
From Lemma~\ref{lemma:bulbesn}, one gets
\begin{gather} \label{eq:rank4}
3\times4(V-E) + 2\times3F_{{\rm int}} - 2(2+\alpha_3)B^3+3\alpha_3B^4\leq 0, \qquad -B^4+4\leq 0.
\end{gather}
Let us multiply the second relation in \eqref{eq:rank4} by a positive rational number $\alpha_4$ such that $- 2(2+\alpha_3)B^3+(3\alpha_3-\alpha_4)B^4$ is an integer. We obtain $-\alpha_4B^4+4\alpha_4\leq0$. Once again, we add this relation to the f\/irst relation in~\eqref{eq:rank4} and obtain
\begin{gather} \label{eq:ran4}
3\times4(V-E) + 2\times3F_{{\rm int}} - 2(2+\alpha_3)B^3+(3\alpha_3-\alpha_4)B^4\leq -4\alpha_4.
\end{gather}
As a remark, the relation \eqref{eq:ran4} is dif\/ferent from the relation~\eqref{eq:forran4} introduced in Section~\ref{subset:polymt}.

We can now f\/ind a general invariant by induction on~$n$.

For all rank $n\geq 4$ w-colored graph, we have
\begin{gather}
\frac{n!}{2}(V-E) + (n-1)!F_{{\rm int}} - (n-2)!(2+\alpha_3)B^3 \nonumber\\
\qquad{} +\sum_{k=4}^n (n-k+1)!\big[(k-1)\alpha_{k-1}-\alpha_k\big]B^k\leq - n\alpha_n.\label{eq:rann}
\end{gather}
The proof of equation \eqref{eq:rann} is by induction on $n$.
Suppose that \eqref{eq:rann} holds for all representative of a rank $n$ w-colored graph. Consider a representative of a rank $n+1$ w-colored graph and let us prove that
\begin{gather*} 
\frac{(n+1)!}{2}(V-E) + n!F_{{\rm int}} - (n-1)!(2+\alpha_3)B^3 \nonumber\\ 
\qquad{}+\sum_{k=4}^{n+1} (n-k+2)!\big[(k-1)\alpha_{k-1}-\alpha_k\big]B^k \leq - (n+1)\alpha_{(n+1)}.
\end{gather*}
The $(n+1)$-bubbles of $\mathcal{G}$ are rank $n$ graphs. Each of them satisf\/ies~\eqref{eq:rann} and a summation on all the $(n+1)$-bubbles of $\mathcal{G}$ gives
\begin{gather*}
\frac{n!}{2}\sum_{{\mathbf{b}}^{n+1} \in \mathcal{B}^{n+1}}(V_{{\mathbf{b}}^{n+1}}-E_{{\mathbf{b}}^{n+1}}) + (n-1)!F_{{\rm int};{\mathbf{b}}^{n+1}} - (n-2)!(2+\alpha_3)\sum_{{\mathbf{b}}^{n+1} \in \mathcal{B}^{n+1}}B^3_{{\mathbf{b}}^{n+1}} \nonumber\\
\qquad{} + \sum_{k=4}^n (n-k+1)!\big[(k-1)\alpha_{k-1}-\alpha_k\big]\sum_{{\mathbf{b}}^{n+1} \in \mathcal{B}^{n+1}}B^k_{{\mathbf{b}}^{n+1}} + n\alpha_nB^{n+1}\leq 0.
\end{gather*}
From Lemma \ref{lemma:bulbesn}, we have
\begin{gather}
(n+1)\frac{n!}{2}(V-E) + n(n-1)!F_{{\rm int}} - (n-1)(n-2)!(2+\alpha_3)B^3\nonumber\\
\qquad{} +\sum_{k=4}^n (n-k+2)(n-k+1)!\big[(k-1)\alpha_{k-1}-\alpha_k\big]B^k + n\alpha_nB^{n+1}\leq 0,\nonumber
 \\ -B^{n+1}+n+1\leq 0.\label{eq:rannpp}
\end{gather}

Let us multiply the second relation in \eqref{eq:rannpp} by a positive rational number $\alpha_{n+1}$ such as
$ - (n-1)(n-2)!(2+\alpha_3)B^3 +\sum\limits_{k=4}^n (n-k+2)(n-k+1)!\big[(k-1)\alpha_{k-1}-\alpha_k\big]B^k
 + (n\alpha_n-\alpha_{n+1})B^{n+1}$
is an integer.
We obtain
\begin{gather}\label{eq:add}
 -\alpha_{n+1}B^{n+1}+(n+1)\alpha_{n+1}\leq0.
\end{gather}
Adding \eqref{eq:add} to the f\/irst relation in \eqref{eq:rannpp} ends the proof.

We could have introduced a polynomial invariant using the expression~\eqref{eq:rannpp} and also it will satisfy similar properties as established in Theorem~\ref{theo:contensn}.

\subsection*{Acknowledgements}

Numerous discussions with Joseph Ben Geloun and Mahouton N.~Hounkonnou have been hugely benef\/icial for this work and gratefully acknowledged. The author acknowledges the support of Max-Planck Institute for Gravitational Physics, Albert Einstein Institute, and the Association pour la Promotion Scientif\/ique de l'Afrique. The ICMPA is also in partnership with the Daniel Iagolnitzer Foundation (DIF), France.

\vspace{-1mm}

\pdfbookmark[1]{References}{ref}
\LastPageEnding

\end{document}